\documentclass{amsart}

\makeatletter
\@namedef{subjclassname@2020}{%
  \textup{2020} Mathematics Subject Classification}
\makeatother 

\usepackage{amsthm,amssymb,amsfonts,latexsym,mathtools,thmtools,amsmath}
\usepackage[T1]{fontenc}
\usepackage{tikz-cd} 
\usepackage{enumitem} 
\usepackage{mathrsfs} 
\usepackage{hyperref} 
\usepackage{hyperref}
\hypersetup{
    colorlinks=true,
    linkcolor=blue,
    filecolor=blue,      
    urlcolor=cyan,
    linktocpage=true
}

\newtheorem{theorem}{Theorem}[section]

\theoremstyle{definition}
\newtheorem{definition}[theorem]{Definition}
\newtheorem{proposition}[theorem]{Proposition}
\newtheorem{corollary}[theorem]{Corollary}
\newtheorem{remark}[theorem]{Remark}
\newtheorem{example}[theorem]{Example}

\theoremstyle{remark}

\numberwithin{equation}{section}

\begin{document}

\title[Weak Annihilators and Nilpotent Associated Primes]{On Weak Annihilators and Nilpotent Associated Primes of Skew PBW Extensions}



\author{Sebasti\'an Higuera}
\address{Universidad Nacional de Colombia - Sede Bogot\'a}
\curraddr{Campus Universitario}
\email{sdhiguerar@unal.edu.co}
\thanks{}

\author{Armando Reyes}
\address{Universidad Nacional de Colombia - Sede Bogot\'a}
\curraddr{Campus Universitario}
\email{mareyesv@unal.edu.co}

\thanks{This work was supported by Faculty of Science, Universidad Nacional de Colombia - Sede Bogot\'a, Bogot\'a, D. C., Colombia [grant number 52464].}



\date{}

\dedicatory{Dedicated to the memory of Professor V. A. Artamonov}

\begin{abstract}

We investigate the notions of weak annihilator and nilpotent associated prime defined by Ouyang and Birkenmeier \cite{OuyangBirkenmeier2012} in the setting of noncommutative rings having PBW bases. We extend several results formulated in the literature concerning annihilators and associated primes of commutative rings and skew polynomial rings to a more general setting of algebras not considered before. We exemplify our results with families of algebras appearing in the theory of enveloping algebras, differential operators on noncommutative spaces, noncommutative algebraic geometry, and theoretical physics. Finally, we present some ideas for future research.

\end{abstract}

\maketitle


\section{Introduction}\label{introduction}
Throughout the paper, every ring is associative (not necessarily commutative) with identity unless otherwise stated. For a ring $R$, $P(R)$, $N(R)$, and $N^{*}(R)$ denote the prime radical of $R$, the set of all nilpotent elements of $R$, and the upper radical $N^{*}(R)$ of $R$ (i.e., the sum of all its nil ideals of $R$), respectively. It is well-known that $N^{*}(R) \subseteq N(R)$, and if the equality $N^{*}(R) = N(R)$ holds, Marks \cite{Marks2001} called $R$ an {\em NI} ring. Due to Birkenmeier et al. \cite{Birkenmeieretal1993}, $R$ is called {\em 2-primal} if $P(R) = N(R)$ (Sun \cite{Sun1991} used the term \emph{weakly symmetric} for these rings). Notice that every reduced ring (a ring which has no non-zero nilpotent elements) is 2-primal. The importance of 2-primal rings is that they can be considered as a generalization of commutative rings and reduced rings. For more details about 2-primal rings, see Birkenmeier et al. \cite{BirkenmeierParkRizvi2013} and Marks \cite{Marks2003}. 

\medskip

For a subset $B$ of a ring $R$, the sets $r_R(B) = \{r\in R\mid Br = 0\}$ and $l_R(B) = \{r\in R\mid rB = 0\}$ represent the right and left annihilator of $B$ in $R$, respectively. Properties of annihilators of subsets in rings have been investigated by several authors \cite{BeachyBlair1975, Faith1989, Faith1991, Kaplansky1965,   Zelmanowitz1976}. For example, Kaplansky \cite{Kaplansky1965} introduced the {\em Baer rings} as those rings for which the right (left) annihilator of every nonempty subset of the ring is generated by an idempotent element. This concept has its roots in functional analysis, having close links to $C^*$-algebras and von Neumann algebras (e.g., Berberian \cite{Berberian2010}).  Closely related to Baer rings are the quasi-Baer rings. A ring $R$ is said to be {\em quasi-Baer} if the left annihilator of an ideal of $R$ is generated by an idempotent. Now, $R$ is called {\em right} ({\em left}) {\em p.p.} ({\em principally projective}) or {\em left Rickart ring} if the right (left) annihilator of each element of $R$ is generated by an idempotent (equivalently, any principal left ideal of the ring is projective). Birkenmeier \cite{Birkenmeieretal2001} defined a ring $R$ to be {\em right} ({\em left}) {\em principally} {\em quasi-Baer} (or simply {\em right} ({\em left}) {\em p.q.-Baer}) ring if the right annihilator of each principal right (left) ideal of $R$ is generated (as a right ideal) by an idempotent. Families of rings satisfying Baer, quasi-Baer, p.p. and p.q.-Baer conditions have been presented in the literature. For instance, Armendariz \cite{Armendariz1974} established that if $R$ is a reduced ring, then the commutative polynomial ring $R[x]$ is a Baer ring if and only if $R$ is a Baer ring \cite[Theorem B]{Armendariz1974}. Birkenmeier et al. \cite{Birkenmeieretal2001} showed that the quasi-Baer condition is preserved by many polynomial extensions and proved that a ring $R$ is right p.q.-Baer if and only if $R[x]$ is right p.q.-Baer. For more details about these ring-theoretical properties, commutative and noncommutative examples, see \cite{Birkenmeieretal2001, Hashemietal2003, HashemiMoussavi2005, HongKimKwak2000, ReyesSuarezUMA2018}, and references therein, and the excellent treatment by Birkenmeier et al. \cite{BirkenmeierParkRizvi2013}.

\medskip

As a generalization of right and left annihilators, Ouyang and Birkenmeier \cite{OuyangBirkenmeier2012} introduced the notion of weak annihilator of a subset in a ring, and investigated its properties over Ore extensions (also known as skew polynomial rings) $R[x;\sigma,\delta]$ defined by Ore \cite{Ore1933}. More exactly, for $R$ a ring and $B$ a subset of $R$, they defined $N_R(B) = \{r\in R\mid br \in N(R),\ {\rm for\ all}\ b\in B\}$, which is called the {\em weak annihilator of} $B$ {\em in} $R$ \cite[Definition 2.1]{OuyangBirkenmeier2012}. The theory developed by them makes use of the notion of compatible ring in the sense of Annin \cite{Annin2004} (see also Hashemi and Moussavi \cite{HashemiMoussavi2005}). Briefly, if $R$ is a ring, $\sigma$ is an endomorphism of $R$, and $\delta$ is a $\sigma$-derivation of $R$, then (i) $R$ is said to be $\sigma$-{\em compatible} if for each $a, b\in R$, $ab = 0$ if and only if $a\sigma(b)=0$ (necessarily, the endomorphism $\sigma$ is injective). (ii) $R$ is called $\delta$-{\em compatible} if for each $a, b\in R$, $ab = 0$ implies $a\delta(b)=0$. (iii) If $R$ is both $\sigma$-compatible and $\delta$-compatible, then $R$ is called ($\sigma,\delta$)-{\em compatible}. 

\medskip

For $R$ a $(\sigma,\delta)$-compatible 2-primal ring, Ouyang and Birkenmeier \cite{OuyangBirkenmeier2012} proved the following: (i) if for each subset $B\nsubseteq N(R)$, $N_R(B)$ is generated as an ideal by a nilpotent element, then for each subset $U\nsubseteq N(R[x; \sigma, \delta])$, $N_{R[x;\sigma,\delta]}(U)$ is generated as an ideal by a nilpotent element \cite[Theorem 2.1]{OuyangBirkenmeier2012}. (ii) The following two statements are equivalent: (1) for each subset $B\nsubseteq N(R)$, $N_R(B)$ is generated as an ideal by a nilpotent element. (2) For each subset $U\nsubseteq N(R[x;\sigma])$, $N_{R[x;\sigma]}(U)$ is generated as an ideal by a nilpotent element \cite[Theorem 2.2]{OuyangBirkenmeier2012}. (iii) If for each principal right ideal $p R\nsubseteq N(R)$, $N_R(p R)$ is generated as an ideal by a nilpotent element, then for each principal right ideal $f(x)  R[x;\sigma,\delta] \nsubseteq N(R[x;\sigma,\delta])$, $N_{R[x;\sigma,\delta]}(f(x) R[x;\sigma,\delta])$ is generated as an ideal by a nilpotent element \cite[Theorem 2.3]{OuyangBirkenmeier2012}. (iv) The following statements are equivalent: (1) For each principal right ideal $p R\nsubseteq N(R)$, $N_R(p R)$ is generated as an ideal by a nilpotent element. (2) For each principal right ideal $f(x)  R[x;\sigma] \nsubseteq N(R)$, $N_{R[x;\sigma]} (f(x) R[x;\sigma])$ is generated as an ideal by a nilpotent element \cite[Theorem 2.4]{OuyangBirkenmeier2012}. (v) If for each $p\notin N(R)$, $N_R(p)$ is generated as an ideal by a nilpotent element, then for each $f(x)\notin N(R[x;\sigma,\delta])$, $N_{R[x;\sigma, \delta]} (f(x))$ is generated as an ideal by a nilpotent element \cite[Theorem 2.5]{OuyangBirkenmeier2012}. (vi) The following statements are equivalent: (1) for each $p \notin N_R(p)$, $N_R(p)$ is generated as an ideal by a nilpotent element. (2) For each skew polynomial $f(x)\notin N(R[x;\sigma])$, $N_{R[x;\sigma]}(f(x))$ is generated as an ideal by a nilpotent element.

\medskip

With the aim of investigating if the previous results obtained by Ouyang and Birkenmeier hold in noncommutative polynomial extensions more general than Ore extensions, in this paper we consider the class of skew Poincar\'e-Birkhoff-Witt extensions introduced by Gallego and Lezama \cite{GallegoLezama2011} because they include families of rings appearing in noncommutative algebra and noncommutative geometry (see Section \ref{SPBW} for a description of the generality of these objects with respect to another families of noncommutative rings). Ring-theoretical, homological and geometrical properties of these objects have been investigated by some people (e.g. \cite{Artamonov2015, LFGRSV, HashemiKhalilAlhevaz2019, HigueraReyes2022, Lezama2020, LezamaVenegas2020, SuarezChaconReyes2022, Tumwesigyeetal2020}).  

\medskip


The paper is organized as follows. In Section \ref{SPBW}, we recall some definitions and results about skew PBW extensions and compatible rings for finite families of endomorphisms and derivations of rings. Section \ref{WA} contains the original results of the article concerning weak annihilator ideals of skew PBW extensions (Theorems \ref{Theorem3.4}, \ref{theorem3.4}, \ref{theorem3.4.1}, \ref{Theorem3.1.6}, \ref{Theorem3.13}, \ref{theorem3.9}, and \ref{Theorem3.18}). Next, in Section \ref{NAP}, we characterize the nilpotent associated prime ideals of skew PBW extensions (Theorems \ref{proposition3.4} and \ref{Nilpotentprimes}). As expected, our results generalize those above corresponding skew polynomial extensions presented by Ouyang et al. \cite{OuyangBirkenmeier2012, OuyangLiu2012}. It is worth mentioning that this work is a sequel of the study of ideals of skew PBW extensions that has been realized by different authors (e.g. \cite{HashemiKhalilAlhevaz2019, LezamaAcostaReyes2015, LouzariReyes2020,  NinoReyes2020, ReyesRodriguez2021, ReyesSuarez2021, ReyesSuarezYesica2018}). In this way, the results formulated in this paper about associated primes extend or contribute to those presented by Annin \cite{Annin, Annin2002, Annin2004}, Bhat \cite{Bhat2008, Bhat2010}, Brewer and Heinzer \cite{BrewerHeinzer1974}, Faith \cite{Faith1991}, Leroy and Matczuk \cite{LeroyMatczuk2004}, and  Ni\~no et al. \cite{NinoRamirezReyes2020}. Finally, Section \ref{examplespaper} illustrate the results established in Sections \ref{WA} and \ref{NAP} with several noncommutative algebras that cannot be expressed as Ore extensions.

\medskip

Throughout the paper, $\mathbb{N}$, $\mathbb{Z}$, $\mathbb{R}$, and $\mathbb{C}$ denote the classical numerical systems. We assume the set of natural numbers including zero. The symbol $\Bbbk$ denotes a field and $\Bbbk^{*} := \Bbbk\ \backslash\ \{0\}$.

\section{Skew Poincar\'e-Birkhoff-Witt extensions}\label{SPBW}
Skew PBW extensions were defined by Gallego and Lezama \cite{GallegoLezama2011} with the aim of generalizing Poincar\'e-Birkhoff-Witt extensions introduced by Bell and Goodearl \cite{BellGoodearl1988} and Ore extensions of injective type defined by Ore in \cite{Ore1933}. Over the years, several authors have shown that skew PBW extensions also generalize families of noncommutative algebras such as 3-dimensional skew polynomial algebras introduced by Bell and Smith \cite{BellSmith1990}, diffusion algebras defined by Isaev et al. \cite{IPR01}, ambiskew polynomial rings introduced by Jordan in several papers \cite{Jordan1993, Jordan1993b, Jordan1995, Jordan1995b, Jordan2000,JordanWells1996}, solvable polynomial rings introduced by Kandri-Rody and Weispfenning  \cite{KandryWeispfenninig1990}, almost normalizing extensions defined by McConnell and Robson \cite{McConnellRobson2001}, skew bi-quadratic algebras recently introduced by Bavula \cite{Bavula2021}, and others (see \cite {LFGRSV} for more details). 

\medskip

The importance of skew PBW extensions is that they do not assume that the coefficients commute with the variables, and the coefficients do not necessarily belong to fields (see Definition \ref{gpbwextension}). As a matter of fact, skew PBW extensions contain well-known groups of algebras such as some types of $G$-algebras in the sense of Apel \cite{Apel1988}, Auslander-Gorenstein rings, some Calabi-Yau and skew Calabi-Yau algebras, some Artin-Schelter regular algebras, some Koszul algebras, quantum polynomials, some quantum universal enveloping algebras, families of differential operator rings, and many other algebras of great interest in noncommutative algebraic geometry and noncommutative differential geometry. For more details about relations between skew PBW extensions and other noncommutative algebras having PBW bases, see \cite{BuesoTorrecillasVerschoren2003, LFGRSV, GomezTorrecillas2014, LezamaVenegas2020, Seiler2010}, and references therein.

\begin{definition}[{\cite[Definition 1]{GallegoLezama2011}}]\label{gpbwextension}
Let $R$ and $A$ be rings. We say that $A$ is a {\em skew PBW extension} (also known as {\em $\sigma$-PBW extension}) {\em over}  $R$, which is denoted by $A:=\sigma(R)\langle
x_1,\dots,x_n\rangle$, if the following conditions hold:
\begin{enumerate}
\item[\rm (i)] $R$ is a subring of $A$ sharing the same multiplicative identity element.
\item[\rm (ii)]there exist elements $x_1,\dots ,x_n\in A$ such that $A$ is a left free $R$-module with basis given by ${\rm Mon}(A):= \{x^{\alpha}=x_1^{\alpha_1}\cdots
x_n^{\alpha_n}\mid \alpha=(\alpha_1,\dots ,\alpha_n)\in
\mathbb{N}^n\}$, where $x_1^{0}\dotsb x_n^{0}:=1$.
\item[\rm (iii)]For each $1\leq i\leq n$ and any $r\in R\ \backslash\ \{0\}$, there exists an element $c_{i,r}\in R\ \backslash\ \{0\}$ such that $x_ir-c_{i,r}x_i\in R$.
\item[\rm (iv)] For any elements $1\leq i,j\leq n$, there exists $d_{i,j}\in R\ \backslash\ \{0\}$ such that $x_jx_i-d_{i,j}x_ix_j\in R+Rx_1+\cdots +Rx_n$ (i.e., there exist elements $r_0^{(i,j)}, r_1^{(i,j)}, \dotsc, r_n^{(i,j)}$ of $R$ with $x_jx_i - d_{i,j}x_ix_j = r_0^{(i,j)} + \sum_{k=1}^{n} r_k^{(i,j)}x_k$).
\end{enumerate}
\end{definition}

From Definition \ref{SPBW} it follows that every non-zero element $f \in A$ can be uniquely expressed as $f = a_0X_0 + a_1X_1 + \cdots + a_mX_m$, with $a_i \in R$ and $X_i \in \text{Mon}(A)$, for $0 \leq i \leq m$ ($X_0:=1$) \cite[Remark 2]{GallegoLezama2011}. 

\medskip

Thinking about Ore extensions \cite{Ore1933}, if $A=\sigma(R)\langle x_1,\dots,x_n\rangle$ is a skew PBW extension over $R$, then for every $1\leq i\leq n$ there exist an injective endomorphism $\sigma_i:R\rightarrow R$ and a $\sigma_i$-derivation $\delta_i:R\rightarrow R$ such that $x_ir=\sigma_i(r)x_i+\delta_i(r)$, for each $r\in R$ \cite[Proposition 3]{GallegoLezama2011}. When necessary, we will write  $\Sigma:=\{\sigma_1,\dotsc, \sigma_n\}$ and $\Delta:=\{\delta_1,\dotsc, \delta_n\}$.

\begin{definition}\label{sigmapbwderivationtype}
Let $A=\sigma(R)\langle
x_1,\dots,x_n\rangle$ be a skew PBW extension over $R$.
\begin{enumerate}
\item[\rm (i)] (\cite[Definition 4]{GallegoLezama2011}) $A$ is called \textit{quasi-commutative} if the conditions {\rm(}iii{\rm)} and {\rm(}iv{\rm)} in Definition \ref{gpbwextension} are replaced by the following: (iii') for each $1\leq i\leq n$ and all $r\in R\ \backslash\ \{0\}$, there exists $c_{i,r}\in R\ \backslash\ \{0\}$ such that $x_ir=c_{i,r}x_i$. (iv') for any $1\leq i,j\leq n$, there exists $d_{i,j}\in R\ \backslash\ \{0\}$ such that $x_jx_i=d_{i,j}x_ix_j$.
\item[\rm (ii)] (\cite[Definition 4]{GallegoLezama2011}) $A$ is called \textit{bijective} if $\sigma_i$ is bijective for each $1\leq i\leq n$, and $d_{i,j}$ is invertible, for any $1\leq i, j\leq n$.
\item[\rm (iii)] (\cite[Definition 2.3]{LezamaAcostaReyes2015}) $A$ is called of {\em endomorphism type} if $\delta_i=0$, for every $i$.  In addition, if every $\sigma_i$ is bijective, then $A$ is said to be a skew PBW extension of {\em automorphism type}.
\end{enumerate}
\end{definition}

Some relations between skew polynomial rings \cite{Ore1933}, PBW extensions \cite{BellGoodearl1988} and skew PBW extensions are presented below.

\begin{remark}\label{comparisonendomorphism}
\begin{itemize}
    \item [\rm (i)] (\cite[Example 5(3)]{LezamaReyes2014}) If $R[x_1;\sigma_1,\delta_1]\dotsb [x_n;\sigma_n,\delta_n]$ is an iterated Ore extension such that $\sigma_i$ is injective for $1\le i\le n$; $\sigma_i(r)$, $\delta_i(r)\in R$, for $r\in R$ and $1\le i\le n$; $\sigma_j(x_i)=c_ix_i+d_i$, for $i < j$, $c_i, d_i\in R$, where $c$ has a left inverse; $\delta_j(x_i)\in R + Rx_1 + \dotsb + Rx_n$, for $i < j$, then $R[x_1;\sigma_1,\delta_1]\dotsb [x_n;\sigma_n, \delta_n] \cong \sigma(R)\langle x_1,\dotsc, x_n\rangle$. 
    \item [\rm (ii)] Skew polynomial rings are not included in PBW extensions. For instance, the quantum plane defined as the quotient $\Bbbk\langle x, y\rangle / \langle xy - qyx\mid q\in \Bbbk^{*} \rangle$ is an Ore extension of injective type given by $\Bbbk[y][x;\sigma]$, where $\sigma(y) = qy$, but it is clear that this cannot be expressed as a PBW extension over $\Bbbk$ or $\Bbbk[y]$. 
    \item [\rm (iii)] Skew PBW extensions are not contained in skew polynomial rings. For example, the universal enveloping algebra $U(\mathfrak{g})$ of a Lie algebra $\mathfrak{g}$ is a PBW extension (\cite[Section 5]{BellGoodearl1988}), and hence a skew PBW extension, but in general it cannot be expressed as an Ore extension.  
     \item [\rm (iii)] (\cite[Theorem 2.3]{LezamaReyes2014}) If $A=\sigma(R)\langle x_1,\dotsc, x_n\rangle$ is a quasi-commutative skew PBW extension over a ring $R$, then $A$ is isomorphic to an iterated Ore extension of endomorphism type (i.e., that is, $R[x;\sigma]$ with $\sigma$ surjective).
    \item [\rm (iv)] (\cite[Remark 3.4]{HigueraReyes2022}) Skew PBW extensions of endomorphism type are more general than iterated Ore extensions of endomorphism type. Let us illustrate the situation with two and three indeterminates.
		
	For the iterated Ore extension of endomorphism type $R[x;\sigma_x][y;\sigma_y]$, if $r\in R$ then we have the following relations: $xr = \sigma_x(r)x$, $yr = \sigma_y(r)y$, and $yx = \sigma_y(x)y$. Now, if we have $\sigma(R)\langle x, y\rangle$ a skew PBW extension of endomorphism type over $R$, then for any $r\in R$, Definition \ref{gpbwextension} establishes that $xr=\sigma_1(r)x$, $yr=\sigma_2(r)y$, and $yx = d_{1,2}xy + r_0 + r_1x + r_2y$, for some elements $d_{1,2}, r_0, r_1$ and $r_2$ belong to $R$. From these relations it is clear which one of them is more general.
	
If we have the iterated Ore extension $R[x;\sigma_x][y;\sigma_y][z;\sigma_z]$, then for any $r\in R$, $xr = \sigma_x(r)x$, $yr = \sigma_y(r)y$, $zr = \sigma_z(r)z$, $yx = \sigma_y(x)y$, $zx = \sigma_z(x)z$, $zy = \sigma_z(y)z$. For the skew PBW extension of endomorphism type $\sigma(R)\langle x, y, z\rangle$, $xr=\sigma_1(r)x$, $yr=\sigma_2(r)y$, $zr = \sigma_3(r)z$, $yx = d_{1,2}xy + r_0 + r_1x + r_2y + r_3z$, $zx = d_{1,3}xz + r_0' + r_1'x + r_2'y + r_3'z$, and $zy = d_{2,3}yz + r_0'' + r_1''x + r_2''y + r_3''z$, for some elements $d_{1,2}, d_{1,3}, d_{2,3}, r_0, r_0', r_0'', r_1, r_1', r_1'', r_2, r_2', r_2'', r_3$, $r_3', r_3''$ of $R$. As the number of indeterminates increases, the differences between both algebraic structures are more remarkable.
\end{itemize}
\end{remark}

Following \cite[Section 3]{GallegoLezama2011}, if $A = \sigma(R)\langle x_1,\dots,x_n\rangle$ is a skew PBW extension over $R$, then we consider the following notation:
\begin{enumerate}
\item[\rm (i)] For the families $\Sigma$ and $\Delta$ in Definition \ref{gpbwextension}, from now on, if $\alpha=(\alpha_1,\dots,\alpha_n)$ belongs to $\mathbb{N}^n$, we will write $\sigma^{\alpha}:=\sigma_1^{\alpha_1}\circ \dotsb \circ \sigma_n^{\alpha_n}$,  $\sigma^{-\alpha}:=\sigma_n^{\alpha_n}\circ \dotsb \circ \sigma_1^{\alpha_1}$ and $\delta^{\alpha} = \delta_1^{\alpha_1} \circ \dotsb \circ \delta_n^{\alpha_n}$, where $\circ$ denotes the classical composition of functions, and $|\alpha|:=\alpha_1+\cdots+\alpha_n$. If $\beta=(\beta_1,\dots,\beta_n)\in \mathbb{N}^n$, then
$\alpha+\beta:=(\alpha_1+\beta_1,\dots,\alpha_n+\beta_n)$.
\item[\rm (ii)] Let $\succeq$ be a total order defined on ${\rm Mon}(A)$. If $x^{\alpha}\succeq x^{\beta}$ but $x^{\alpha}\neq x^{\beta}$, we will write $x^{\alpha}\succ x^{\beta}$. As we saw above, if $f$ is a non-zero element of $A$, then $f$ can be expressed uniquely as $f=a_0X_0 + a_1X_1+\dotsb +a_mX_m$, with $a_i\in R$, $X_m\succ \dotsb \succ X_1$, and $X_0 := 1$. Eventually, we use expressions as $f=a_0Y_0 + a_1Y_1+\dotsb +a_mY_m$, with $a_i\in R$, and $Y_m\succ \dotsb \succ Y_1$. We define ${\rm
lm}(f):=X_m$, the \textit{leading monomial} of $f$; ${\rm
lc}(f):=a_m$, the \textit{leading coefficient} of $f$; ${\rm
lt}(f):=a_mX_m$, the \textit{leading term} of $f$; ${\rm exp}(f):={\rm exp}(X_m)$, the \textit{order} of $f$. Notice that $\deg(f):={\rm max}\{\deg(X_i)\}_{i=1}^m$. If $f=0$, then
${\rm lm}(0):=0$, ${\rm lc}(0):=0$, ${\rm lt}(0):=0$. Finally, we also
consider $X\succ 0$ for any $X\in {\rm Mon}(A)$, and hence we extend $\succeq$ to ${\rm Mon}(A)\cup \{0\}$.

From \cite[Definition 11]{GallegoLezama2011}, if $\succeq$ is a total order on ${\rm Mon}(A)$, then we say that $\succeq$ is a {\em monomial order} on ${\rm Mon}(A)$ if the following conditions hold:

\begin{itemize}
\item For every $x^{\beta}, x^{\alpha}, x^{\gamma}, x^{\lambda}\in {\rm Mon}(A)$, the relation $x^{\beta}\succeq x^{\alpha}$ implies ${\rm lm}(x^{\gamma}x^{\beta}x^{\lambda}) \succeq {\rm lm}(x^{\gamma}x^{\alpha}x^{\lambda})$ (i.e., the total order is compatible with multiplication). 
\item $x^{\alpha}\succeq 1$, for every $x^{\alpha}\in {\rm Mon}(A)$.
\item $\succeq$ is degree compatible, i.e., $|\beta| \succeq |\alpha|\Rightarrow x^{\beta}\succeq x^{\alpha}$.
\end{itemize}

Monomial orders are also called {\em admissible orders}. The condition (iii) of the previous definition is needed in the proof of the fact that every monomial order on ${\rm Mon}(A)$ is a well order, that is, there are not infinite decreasing chains in ${\rm Mon}(A)$ \cite[Proposition 12]{GallegoLezama2011}. 
\item [\rm (iii)] For a skew PBW extension $A = \sigma(R)\langle x_1,\dotsc, x_n\rangle$ over a ring $R$ and $B$ a subset of $R$, $BA$ denotes the set $\bigl\{f = \sum_{i=0}^{m} b_iX_i\mid b_i \in B,\ {\rm for\ all}\ i\bigr\}$.
\end{enumerate}

Propositions \ref{coefficientes} and \ref{Reyes2015Proposition2.9} are very useful in the calculations presented in Sections \ref{WA} and \ref{NAP}. 

\begin{proposition}[{\cite[Theorem 7]{GallegoLezama2011}}]\label{coefficientes}
If $A$ is a polynomial ring with coefficients in $R$ with respect to the set of indeterminates $\{x_1,\dots,x_n\}$, then $A=\sigma(R)\langle x_1,\dots,x_n\rangle$ is a skew PBW extension over $R$ if and only if the following conditions hold:
\begin{enumerate}
\item[\rm (1)]for each $x^{\alpha}\in {\rm Mon}(A)$ and every $0\neq r\in R$, there exist unique elements $r_{\alpha}:=\sigma^{\alpha}(r)\in R\ \backslash\ \{0\}$, $p_{\alpha ,r}\in A$, such that $x^{\alpha}r=r_{\alpha}x^{\alpha}+p_{\alpha, r}$,  where $p_{\alpha ,r}=0$, or $\deg(p_{\alpha ,r})<|\alpha|$ if $p_{\alpha , r}\neq 0$. If $r$ is left invertible,  so is $r_\alpha$.
\item[\rm (2)]For each $x^{\alpha},x^{\beta}\in {\rm Mon}(A)$,  there exist unique elements $d_{\alpha,\beta}\in R$ and $p_{\alpha,\beta}\in A$ such that $x^{\alpha}x^{\beta}=d_{\alpha,\beta}x^{\alpha+\beta}+p_{\alpha,\beta}$, where $d_{\alpha,\beta}$ is left invertible, $p_{\alpha,\beta}=0$, or $\deg(p_{\alpha,\beta})<|\alpha+\beta|$ if $p_{\alpha,\beta}\neq 0$.
\end{enumerate}
\end{proposition}

\begin{proposition}[{\cite[Proposition 2.7 and Remark 2.8]{ReyesRodriguez2021}}]\label{Reyes2015Proposition2.9}
Let $A=\sigma(R)\langle x_1,\dotsc, x_n\rangle$ be a skew PBW extension over $R$. If $\alpha=(\alpha_1,\dotsc, \alpha_n)\in \mathbb{N}^{n}$ and $r$ is an element of $R$, then  
{\footnotesize{\begin{align*}
x^{\alpha}r = &\ x_1^{\alpha_1}x_2^{\alpha_2}\dotsb x_{n-1}^{\alpha_{n-1}}x_n^{\alpha_n}r = x_1^{\alpha_1}\dotsb x_{n-1}^{\alpha_{n-1}}\biggl(\sum_{j=1}^{\alpha_n}x_n^{\alpha_{n}-j}\delta_n(\sigma_n^{j-1}(r))x_n^{j-1}\biggr)\\
+ &\ x_1^{\alpha_1}\dotsb x_{n-2}^{\alpha_{n-2}}\biggl(\sum_{j=1}^{\alpha_{n-1}}x_{n-1}^{\alpha_{n-1}-j}\delta_{n-1}(\sigma_{n-1}^{j-1}(\sigma_n^{\alpha_n}(r)))x_{n-1}^{j-1}\biggr)x_n^{\alpha_n}\\
+ &\ x_1^{\alpha_1}\dotsb x_{n-3}^{\alpha_{n-3}}\biggl(\sum_{j=1}^{\alpha_{n-2}} x_{n-2}^{\alpha_{n-2}-j}\delta_{n-2}(\sigma_{n-2}^{j-1}(\sigma_{n-1}^{\alpha_{n-1}}(\sigma_n^{\alpha_n}(r))))x_{n-2}^{j-1}\biggr)x_{n-1}^{\alpha_{n-1}}x_n^{\alpha_n}\\
+ &\ \dotsb + x_1^{\alpha_1}\biggl( \sum_{j=1}^{\alpha_2}x_2^{\alpha_2-j}\delta_2(\sigma_2^{j-1}(\sigma_3^{\alpha_3}(\sigma_4^{\alpha_4}(\dotsb (\sigma_n^{\alpha_n}(r))))))x_2^{j-1}\biggr)x_3^{\alpha_3}x_4^{\alpha_4}\dotsb x_{n-1}^{\alpha_{n-1}}x_n^{\alpha_n} \\
+ &\ \sigma_1^{\alpha_1}(\sigma_2^{\alpha_2}(\dotsb (\sigma_n^{\alpha_n}(r))))x_1^{\alpha_1}\dotsb x_n^{\alpha_n}, \ \ \ \ \ \ \ \ \ \ \sigma_j^{0}:={\rm id}_R\ \ {\rm for}\ \ 1\le j\le n.
\end{align*}}}
If $a_i, b_j\in R$ and $X_i:=x_1^{\alpha_{i1}}\dotsb x_n^{\alpha_{in}},\ Y_j:=x_1^{\beta_{j1}}\dotsb x_n^{\beta_{jn}}$, when we compute every summand of $a_iX_ib_jY_j$ we obtain pro\-ducts of the coefficient $a_i$ with several evaluations of $b_j$ in $\sigma$'s and $\delta$'s depending on the coordinates of $\alpha_i$. This assertion follows from the expression:

\begin{align*}
a_iX_ib_jY_j = &\ a_i\sigma^{\alpha_{i}}(b_j)x^{\alpha_i}x^{\beta_j} + a_ip_{\alpha_{i1}, \sigma_{i2}^{\alpha_{i2}}(\dotsb (\sigma_{in}^{\alpha_{in}}(b_j)))} x_2^{\alpha_{i2}}\dotsb x_n^{\alpha_{in}}x^{\beta_j} \\
+ &\ a_i x_1^{\alpha_{i1}}p_{\alpha_{i2}, \sigma_3^{\alpha_{i3}}(\dotsb (\sigma_{{in}}^{\alpha_{in}}(b_j)))} x_3^{\alpha_{i3}}\dotsb x_n^{\alpha_{in}}x^{\beta_j} \\
+ &\ a_i x_1^{\alpha_{i1}}x_2^{\alpha_{i2}}p_{\alpha_{i3}, \sigma_{i4}^{\alpha_{i4}} (\dotsb (\sigma_{in}^{\alpha_{in}}(b_j)))} x_4^{\alpha_{i4}}\dotsb x_n^{\alpha_{in}}x^{\beta_j}\\
+ &\ \dotsb + a_i x_1^{\alpha_{i1}}x_2^{\alpha_{i2}} \dotsb x_{i(n-2)}^{\alpha_{i(n-2)}}p_{\alpha_{i(n-1)}, \sigma_{in}^{\alpha_{in}}(b_j)}x_n^{\alpha_{in}}x^{\beta_j} \\
+ &\ a_i x_1^{\alpha_{i1}}\dotsb x_{i(n-1)}^{\alpha_{i(n-1)}}p_{\alpha_{in}, b_j}x^{\beta_j},
\end{align*}
where the polynomials $p$'s are given by Proposition \ref{coefficientes}.
\end{proposition}

\subsection{\texorpdfstring{$(\Sigma, \Delta)$}\ -compatible and weak \texorpdfstring{$(\Sigma, \Delta)$}\ -compatible rings}\label{sigmarigidsigmadeltacompatible} 
As we said in the Introduction, following Annin \cite{Annin2004} (see also Hashemi and Moussavi \cite{HashemiMoussavi2005}), for an endomorphism $\sigma$ and a $\sigma$-derivation $\delta$ of a ring $R$, $R$ is said to be {\it $\sigma$-compatible} if for every $a, b \in R$, we have $ab = 0$ if and only if $a\sigma(b) = 0$ (necessarily, $\sigma$ is injective). $R$ is called {\it $\delta$-compatible} if for each $a, b \in R$, $ab = 0$ implies $a\delta(b) = 0$. If $R$ is both $\sigma$-compatible and $\delta$-compatible, then $R$ is said to be {\it $(\sigma, \delta)$-compatible}. From \cite[Lemma 3.3]{Hashemietal2003} we know that $(\sigma, \delta)$-compatible rings generalize the $\sigma$-rigid rings (an endomorphism $\sigma$ of a ring $R$ is called {\em rigid} if $r\sigma(r) = 0$ implies $r=0$, where $r\in R$; $R$ is said to be {\em rigid} if there exists a rigid endomorphism $\sigma$ of $R$) introduced by Krempa \cite{Krempa1996}. Thinking in the context of skew PBW extensions, Hashemi et al. \cite{HashemiKhalilAlhevaz2017} and Reyes and Su\'arez \cite {ReyesSuarezUMA2018} introduced independently the $(\Sigma, \Delta)$-{\em compatible rings} as a natural generalization of $(\sigma, \delta)$-compatible rings. Examples, ring and module theoretic properties of these structures can be found in \cite{Hamidizadehetal2020, HashemiKhalilAlhevaz2019, HashemiKhalilGhadiri2019, LouzariReyes2020, Reyes2019, ReyesSuarezSymmetry2021}.

\medskip

Since the notion of $(\sigma, \delta)$-compatible ring was used by Ouyang and Birkenmeier \cite{OuyangBirkenmeier2012} in their treatment about weak annihilators of Ore extensions, it is to be expected that we have to consider the notion of $(\Sigma,\Delta)$-compatible ring to characterize this type of annihilators in the setting of skew PBW extensions. In this way, next we consider a finite family of endomorphisms $\Sigma$ and a finite family $\Delta$ of $\Sigma$-derivations of a ring $R$.

\begin{definition}
[{\cite[Definition 3.1]{HashemiKhalilAlhevaz2017}; \cite[Definition 3.2]{ReyesSuarezUMA2018}}] 
$R$ is said to be {\it $\Sigma$-compatible} if for each $a, b \in R$, $a\sigma^{\alpha}(b) = 0$ if and only if $ab = 0$, where $\alpha \in \mathbb{N}^n$. $R$ is said to be {\it $\Delta$-compatible} if for each $a, b \in R$, it follows that $ab = 0$ implies $a\delta^{\beta}(b)=0$, where $\beta \in \mathbb{N}^n$. If $R$ is both $\Sigma$-compatible and $\Delta$-compatible, then $R$ is called {\it  $(\Sigma, \Delta)$-compatible}.
\end{definition}

The following proposition is the natural generalization of \cite[Lemma 2.1]{HashemiMoussavi2005}.

\begin{proposition}[{\cite[Lemma 3.3]{HashemiKhalilAlhevaz2017}; \cite[Proposition 3.8]{ReyesSuarezUMA2018}}]\label{Proposition3.8}
If $R$ is a $(\Sigma, \Delta)$-compatible ring, then for every elements $a, b$ belong to $R$, we have the following assertions:
\begin{enumerate}
\item [\rm (1)] If $ab=0$, then $a\sigma^{\theta}(b) = \sigma^{\theta}(a)b=0$, where $\theta\in \mathbb{N}^{n}$.
\item [\rm (2)] If $\sigma^{\beta}(a)b=0$, for some $\beta\in \mathbb{N}^{n}$, then $ab=0$.
\item [\rm (3)] If $ab=0$, then $\sigma^{\theta}(a)\delta^{\beta}(b)= \delta^{\beta}(a)\sigma^{\theta}(b) = 0$, where $\theta, \beta\in \mathbb{N}^{n}$.
\end{enumerate}
\end{proposition}

\begin{example}\label{ExampleF4}
Let $\mathbb{F}_4=\left \{0,1,a,a^2 \right \}$ be the field of four elements. Consider the ring of polynomials $\mathbb{F}_4[z]$ and let $R=\frac{\mathbb{F}_4[z]}{\left \langle z^2 \right \rangle}$. For simplicity, we identify the elements of $\mathbb{F}_4[z]$ with their images in $R$. Let $\Sigma=\left \{\sigma_{i,j} \right \}$ be the family of homomorphisms of $R$ defined by $\sigma_{i,j}(a)=a^i$ and $\sigma_{i,j}(z)=a^jz$, for $1 \leq i \leq 2$ and $0\leq j \leq 2$. Consider the skew PBW extension $A=\sigma(R)\left \langle x_{1,0},x_{1,1},x_{1,2},x_{2,0},x_{2,1},x_{2,2} \right \rangle$, where the indeterminates satisfy the commutation relations $x_{i,j}x_{i',j'}=x_{i',j'}x_{i,j}$, for all $1 \leq i,i' \leq 2$ and $0\leq j,j' \leq 2$. On the other hand, for $a^rz \in R$, we have that $x_{i,j}a^rz=\sigma_{i,j}(a^rz)x_{i,j}=(a^r)^ia^jzx_{i,j}=a^{ri+j}zx_{i,j}$, where $a^{ri+j} \in \mathbb{F}_4$, $1\leq i \leq 2$ and $1\leq r,j \leq 2$. This example can be extended to any finite field $\mathbb{F}_{p^n}$ with $p$ a prime number. It is straightforward to show that $R$ is $\Sigma$-compatible, and thus $A$ is a skew PBW extension over a $\Sigma$-compatible ring $R$. 
\end{example}

\begin{example}\label{ExampleMatrix}
 Let $\Bbbk[t]$ be the commutative polynomial ring over the field $\Bbbk$ in the indeterminate $t$. We consider the identity homomorphism  $\sigma$ of $\Bbbk[t]$ and $\delta(t):= 1$. Having in mind that $\Bbbk[t]$ is reduced, we can show that $\Bbbk[t]$ is a $(\sigma,\delta)$-compatible ring. Let $R$ be the ring defined as
\begin{center}
$\displaystyle R=\left \{ \begin{pmatrix} p(t) & q(t)\\0 &p(t) \end{pmatrix} \ \mid p(t),q(t) \in \Bbbk[t] \right \}$.   
\end{center}
The endomorphism $\sigma$ of $\Bbbk[t]$ is extended to the endomorphism $\overline{\sigma}: R \rightarrow R$ by defining $\overline{\sigma}((a_{ij}))=(\sigma(a_{ij}))$, and the $\sigma$-derivation $\delta$ of $\Bbbk[t]$ is also extended to $\overline{\delta}: R \rightarrow R$ by considering $\overline{\delta}((a_{ij}))=(\delta(a_{ij}))$. Notice that the Ore extension $R[x;\overline{\sigma},\overline{\delta}]$ is a skew PBW extension over $R$ which is $(\overline{\sigma},\overline{\delta})$-compatible.
\end{example}

The following definition presents the notion of compatibility for ideals of a ring.

\begin{definition}[{\cite[Definition 3.1]{HashemiKhalilAlhevaz2019}}]
Let $R$ be a $(\Sigma, \Delta)$-compatible ring where, as above, $\Sigma:=\left \{\sigma_1,\dots,\sigma_n\right\}$ is a finite family of endomorphisms of $R$ and $\Delta:=\left \{\delta_1,\dots ,\delta_n\right \}$ is a finite family of $\Sigma$-derivations of $R$.
\begin{enumerate}
    \item[{\rm (i)}] An ideal $I$ of $R$ is called a $\Sigma$-{\em ideal} if $\sigma^{\alpha}(I)\subseteq I$, where $\alpha \in \mathbb{N}^n$. $I$ is a $\Delta$-{\em ideal} if $\delta^{\alpha}(I)\subseteq I$, where $\alpha \in \mathbb{N}^n$. $I$ is $\Sigma$-{\em invariant} if $\sigma^{-\alpha}(I) =I$, where $\alpha \in \mathbb{N}^n$. If $I$ is both $\Sigma$ and $\Delta$-ideal, then $I$ is called a $(\Sigma,\Delta)$-{\em ideal}.
   \item[{\rm (ii)}] For an  ideal $I$ of $R$, we say that $I$ is $\Sigma$-{\em compatible} if for  each $a,b \in R$ and $\alpha \in \mathbb{N}^n$, $ab \in I$ if and only if $a\sigma^{\alpha}(b)\in I$. Moreover, $I$ is said to be $\Delta$-{\em compatible} ideal if for each $a,b\in R$ and $\alpha \in \mathbb{N}^n$, $ab \in I$ implies $a\delta^{\alpha}(b)\in I$. If $I$ is both $\Sigma$-compatible and $\Delta$-compatible, then we say that $I$ is $(\Sigma,\Delta)$-{\em compatible}. 
\end{enumerate}
\end{definition} 

Notice that $R$ is $\Sigma$-compatible (resp., $\Delta$-compatible) if and only if the zero ideal is $\Sigma$-compatible (resp., $\Delta$-compatible) ideal of $R$. If $R$ is a $(\Sigma, \Delta)$-compatible ring and $N(R)$ is an ideal of $R$, then, $N(R)$ is a $(\Sigma, \Delta)$-compatible ideal.

\begin{proposition}[{\cite[Proposition 3.3]{HashemiKhalilAlhevaz2019}}]\label{PropertynilR}
If $I$ is a $(\Sigma,\Delta)$-compatible ideal of $R$ and $a, b$ are elements of $R$, then the following assertions hold:
\begin{enumerate}
    \item[{\rm (1)}] If $ab \in I$, then $a\sigma^{\alpha}(b) \in I$ and $\sigma^{\alpha}(a)b \in I$, for each $\alpha \in \mathbb{N}^n$.
    \item[{\rm (2)}] If $ab \in I$, then $\sigma^{\alpha}(a)\delta^{\beta}(b),\ \delta^{\beta}(a)\sigma^{\alpha}(b) \in I$, for every $\alpha,\beta \in \mathbb{N}^n$.
\item[{\rm (3)}]  If $a\sigma^{\theta}(b)\in I$ or $\sigma^{\theta}(a)b \in I$, for some $\theta \in \mathbb{N}^n$, then $ab \in I$.
\end{enumerate}
\end{proposition}

Reyes and Su\'arez \cite{ReyesSuarez2019-2} defined the concept of {\em weak} $(\Sigma, \Delta)$-{\em compatible ring} as a generalization of $(\Sigma, \Delta)$-compatible rings and {\em weak} $(\sigma, \delta)$-{\em compatible rings} introduced by Ouyang and Liu \cite{LunquenJingwang2011}. For the next definition, consider again a finite family of endomorphisms $\Sigma$ and a finite family $\Delta$ of $\Sigma$-derivations of a ring $R$.

\begin{definition}
[{\cite[Definition 4.1]{ReyesSuarez2019-2}}] $R$ is said to be {\it weak $\Sigma$-compatible} if for each $a, b \in R$, $a\sigma^{\alpha}(b)\in N(R)$ if and only if $ab \in N(R)$, where $\alpha \in \mathbb{N}^n$. $R$ is said to be {\it weak $\Delta$-compatible} if for each $a, b \in R$, $ab \in N(R)$ implies $a\delta^{\beta}(b)\in N(R)$, where $\beta \in \mathbb{N}^n$. If $R$ is both weak $\Sigma$-compatible and weak $\Delta$-compatible, then $R$ is called {\it weak $(\Sigma, \Delta)$-compatible}.
\end{definition}

It is clear that the following result extends Proposition \ref{Proposition3.8}.

\begin{proposition} [{\cite[Proposition 4.2]{ReyesSuarez2019-2}}] If $R$ is a weak $(\Sigma,\Delta)$-compatible ring, then the following assertions hold:
\begin{itemize}
    \item[(1)] If $ab \in N(R)$, then $a\sigma^{\alpha}(b), \sigma^{\beta}(a)b \in N(R)$, for all $\alpha, \beta \in \mathbb{N}^n$.
    \item[(2)] If $\sigma^{\alpha}(a)b \in N(R)$, for some element $\alpha \in \mathbb{N}^n$, then $ab \in N(R)$.
    \item[(3)] If $a\sigma^{\beta}(b) \in N(R)$, for some element $\beta \in \mathbb{N}^n$, then $ab \in N(R)$.
    \item[(4)] If $ab \in N(R)$, then $\sigma^{\alpha}(a)\delta^{\beta}(b), \delta^{\beta}(a)\sigma^{\alpha}(b) \in N(R)$, for $\alpha, \beta \in \mathbb{N}^n$. 
\end{itemize}
\end{proposition}

The next example shows that there exists a weak $(\Sigma, \Delta)$-compatible ring which is not $(\Sigma, \Delta)$-compatible.

\begin{example}\label{ExampleWeak}
Let $R$ be a ring and $M$ an $(R,R)$-bimodule. \textit{The trivial extension of R by M} is defined as the ring $T(R,M):= R\oplus M$ with the usual addition and the multiplication defined as $(r_1,m_1)(r_2,m_2):=(r_1r_2, r_1m_2 + m_1r_2)$, for all $r_1,r_2 \in R$ and $m_1,m_2 \in M$. Notice that $T(R,M)$ is isomorphic to the matrix ring (with the usual matrix operations) of the form $\bigl(\begin{smallmatrix}r & m\\ 0 & r \end{smallmatrix}\bigr)$, where $r\in R$ and $m \in M$. In particular, we call ${\rm S}_2(\mathbb{Z})$ the ring of matrices isomorphic to $T(\mathbb{Z},\mathbb{Z})$.

Consider the ring ${\rm S}_2(\mathbb{Z})$ given by
\begin{center}
    ${\rm S}_2(\mathbb{Z})= \left \{ \begin{pmatrix}a & b\\ 0 & a \end{pmatrix} \mid a,b \in \mathbb{Z} \right \}$.
\end{center}
Let $\sigma_1={\rm id}_{{\rm S}_2(\mathbb{Z})}$ be the identity endomorphism of ${\rm S}_2(\mathbb{Z})$, and let $\sigma_2$ and $\sigma_3$ be the two endomorphisms of ${\rm S}_2(\mathbb{Z})$ defined by
\begin{center}
    $\sigma_2\left ( \begin{pmatrix}a & b\\ 0 & a \end{pmatrix} \right )= \begin{pmatrix}a & -b\\0 & a\end{pmatrix},\quad {\rm and}\quad \sigma_3\left ( \begin{pmatrix}a & b\\ 0 & a \end{pmatrix} \right )= \begin{pmatrix}a & 0\\0 & a\end{pmatrix}$,
\end{center}
respectively. Notice that the set of nilpotent elements of ${\rm S}_2(\mathbb{Z})$ is given by 
\begin{center}
    $N({\rm S}_2(\mathbb{Z}))= \left \{ \begin{pmatrix}0 & b\\ 0 & 0 \end{pmatrix} \mid b \in \mathbb{Z} \right \}$.
\end{center}
It is straightforward to see that ${\rm S}_2(\mathbb{Z})$ is weak $\Sigma$-compatible where $\Sigma = \left \{\sigma_1,\sigma_2,\sigma_3 \right \}$, but ${\rm S}_2(\mathbb{Z})$ is not $\sigma_3$-compatible. For example, if we take the matrices $C,D \in {\rm S}_2(\mathbb{Z})$ given by 
\begin{center}
    $C = \begin{pmatrix} 1 & 1\\ 0& 1\end{pmatrix},\ \ D = \begin{pmatrix} 0 & 1\\ 0 & 0 \end{pmatrix},$
\end{center}
then it is clear that $C\sigma_3(D)=0$ but $CD = \begin{pmatrix} 0 & 1\\ 0&
0\end{pmatrix}\neq 0$, whence ${\rm S}_2(\mathbb{Z})$ is not $\Sigma$-compatible.
\end{example}

Ouyang and Liu \cite[Lemma 2.13]{LunquenJingwang2011} characterized the nilpotent elements in an Ore extension over a weak $(\sigma, \delta)$-compatible and NI ring. Since skew PBW extensions generalize Ore extensions of injective type, Reyes and Su\'arez \cite{ReyesSuarez2019-2} extended this result as the following proposition shows. The NI property for skew PBW extensions was recently studied by Su\'arez et al. \cite{SuarezChaconReyes2022}. 

\begin{proposition}[{\cite[Theorem 4.6]{ReyesSuarez2019-2}}]\label{ReyesSuarez2019-2Theorem4.6} If $A = \sigma(R)\langle x_1,\dotsc, x_n\rangle$ is a skew PBW extension over a weak $(\Sigma,\Delta)$-compatible and NI ring $R$, then $f = a_0 + a_1X_1 + \dotsb + a_mX_m\in N(A)$ if and only if $a_i\in N(R)$, for every $i$.	
\end{proposition}

Notice that Proposition \ref{ReyesSuarez2019-2Theorem4.6} also generalizes the results presented by Ouyang and Birkenmeier for Ore extensions \cite[Lemma 2.6]{OuyangBirkenmeier2012}, and Ouyang and Liu for differential polynomials rings \cite[Lemma 2.12]{OuyangLiu2012}. 
Furthermore, since weak $(\Sigma,\Delta)$-compatible rings contain strictly $(\Sigma,\Delta)$-compatible rings, we have immediately the following corollary for skew PBW extensions over $(\Sigma,\Delta)$-compatible rings.

\begin{corollary}\label{corollary2.16} If $A = \sigma(R)\langle x_1,\dotsc, x_n\rangle$ is a skew PBW extension over a $(\Sigma,\Delta)$-compatible and NI ring $R$, then the following statements hold:
\begin{enumerate}
\item [\rm (1)] $N(A)$ is an ideal of $A$ and $N(A) = N(R)\langle x_1,\dotsc, x_n\rangle$. \label{MSc2.3.20}
\item [\rm (2)] $f = a_0 + a_1X_1 + \dotsb + a_mX_m\in N(A)$ if and only if $a_i\in N(R)$, for every $i = 0, \dotsc, m$. \label{MSc2.3.21}
\end{enumerate}
\end{corollary}

Corollary \ref{corollary2.16} generalizes the corresponding theorem for Ore extensions of injective type over $(\sigma,\delta)$-compatible and $2$-primal rings \cite[Corollary 2.2]{OuyangBirkenmeier2012}, and the result for differential polynomials rings over $\delta$-compatible reversible rings \cite[Corollary 2.13]{OuyangLiu2012}.

\medskip

Ouyang et al. \cite{Lunqunetal2013} introduced the notion of skew $\pi$-Armendariz ring in the following way. If $R$ is a ring with an endomorphism $\sigma$ and a $\sigma$-derivation $\delta$, then $R$ is called {\em skew $\pi$-Armendariz ring} if for polynomials $f(x) = \sum_{i=0}^{l} a_ix^i$ and $g(x) = \sum_{j=0}^{m} b_jx^j$ in $R[x; \sigma, \delta]$, $f(x)g(x) \in N(R[x; \sigma, \delta])$ implies that $a_ib_j \in N(R)$, for each $0 \leq i \leq l$ and $0 \leq j \leq m$. Skew $\pi$-Armendariz rings are more general than skew Armendariz rings when the ring of coefficients is $(\sigma,\delta)$-compatible \cite[Theorem 2.6]{Lunqunetal2013}, and also extend $\sigma$-Armendariz rings defined by Hong et al. \cite{Hongetal2006} considering $\delta$ as the zero derivation. For a detailed description of Armendariz rings in the commutative and noncommutative context, including skew PBW extensions, see \cite{Armendariz1974, BirkenmeierParkRizvi2013, HongKimKwak2000}.

\medskip

Ouyang and Liu \cite{LunquenJingwang2011} showed that if $R$ is a weak $(\sigma,\delta)$-compatible and NI ring, then $R$ is skew $\pi$-Armendariz ring  \cite[Corollary 2.15]{LunquenJingwang2011}. Reyes \cite{Reyes2018} formulated the analogue of skew $\pi$-Armendariz ring in the setting of skew PBW extensions. For a skew PBW extension $A = \sigma(R)\langle x_1,\dotsc, x_n\rangle$ over a ring $R$, $R$ is said to be {\em skew} $\Pi$-{\em Armendariz ring} if for elements $f = \sum_{i=0}^l a_iX_i$ and $g = \sum_{j=0}^m b_jY_j$ belong to $A$, $fg \in N(A)$ implies $a_ib_j \in N(R)$, for each $0 \leq i \leq l$ and $0 \leq j \leq m$. Reyes \cite[Theorem 3.10]{Reyes2018} showed that if $R$ is reversible (following Cohn \cite[p. 641]{Cohn1999}, $R$ is called {\em reversible} if $ab = 0$ implies $ba = 0$, where $a, b\in R$) and $(\Sigma,\Delta)$-compatible, then $R$ is skew $\Pi$-Armendariz ring. Related to this result, the following proposition generalizes \cite[Lemma 2.14]{OuyangLiu2012} and \cite[Corollary 2.3]{OuyangBirkenmeier2012}. 
 
\begin{proposition}[{\cite[Theorem 4.7]{ReyesSuarez2019-2}}]\label{ReyesSuarez2019-2Theorem 4.7}
Let $A = \sigma(R)\langle x_1,\dotsc, x_n\rangle$ be a skew PBW extension over a weak $(\Sigma,\Delta)$-compatible and NI ring $R$. If $f = a_0 + a_1X_1 + \dotsb + a_mX_m\in N(A)$ and $g = b_0 + b_1Y_1 + \dotsb + b_tY_t$ are elements of $A$, then $fg\in N(A)$ if and only if $a_ib_j\in N(R)$, for all $i, j$.
\end{proposition}

\section{Weak Annihilator Ideals}\label{WA}
In this section, we study the weak notion of annihilator introduced by Ouyang and Birkenmeier \cite{OuyangBirkenmeier2012} in the setting of skew PBW extensions.  

\begin{definition}[{\cite[Definition 2.1]{OuyangBirkenmeier2012}}]
Let $R$ be a ring. For $X\subseteq R$, it is defined $ N_R(X)= \left \{a \in R \ | \ xa \in N(R)\ \text{for all} \ x \in X \right \}$, which is called the {\em weak annihilator} of $X$ in $R$. If $X$ is a singleton, say $X= \left \{r \right \}$, we use $N_R(r)$ to denote $N_R(\left \{r \right\})$.
\end{definition}

Notice that for $X\subseteq R$, the sets given by $\left \{ a \in R \ | \ xa \in N(R)\ \text{for all} \ x \in X \right \}$ and $\left \{ b \in R \ | \ bx \in N(R)\ \text{for all} \ x \in X \right \}$ coincide. Moreover, $l_R(X),\ r_R(X) \subseteq N_R(X)$. It is clear that if $R$ is reduced, then $r_R(X) = l_R(X) = N_R(X)$. In addition, if $N(R)$ is an ideal of $R$, then $N_R(X)$ is also an ideal of $R$ \cite[p. 346]{OuyangBirkenmeier2012}.

\begin{example}[{\cite[Example 2.1]{OuyangBirkenmeier2012}}] Let $T_2(\mathbb{Z})$ be the $2\times 2$ upper triangular matrix ring over $\mathbb{Z}$ and $X=\left \{ \bigl(\begin{smallmatrix} 2 & 0\\ 0 & 2 \end{smallmatrix}\bigr) \right \} \subseteq T_2(\mathbb{Z})$. Then $r_{T_2(\mathbb{Z})}(X)=0$ and $N_{T_2(\mathbb{Z})}(X)= \left \{\bigl(\begin{smallmatrix} 0 & m\\ 0 & 0 \end{smallmatrix}\bigr) \ | \ m \in \mathbb{Z}\right \}$, whence $r_{T_2(\mathbb{Z})}(X)\neq N_{T_2(\mathbb{Z})}(X)$. This shows that a weak annihilator is not an immediate generalization of an annihilator.
\end{example}

We recall some properties of weak annihilators.

\begin{proposition}[{\cite[Proposition 2.1]{OuyangBirkenmeier2012}}]\label{Proposition2.1ouyang} If $X$ and $Y$ are subsets of a ring $R$, then we have the following:
\begin{itemize}
    \item[{\rm (1)}] $X \subseteq Y$ implies $N_R(Y) \subseteq N_R(X)$.
    \item[{\rm (2)}] $X \subseteq N_R(N_R(X)))$.
    \item[{\rm (3)}]$N_R(X)= N_R(N_R(N_R(X)))$.
\end{itemize}
\end{proposition}

For a skew PBW extension $A$ over a ring $R$, let ${\rm NAnn}_R(R) := \left \{N_R(U) \ | \ U \subseteq R\right \}$ and ${\rm NAnn}_A(A) := \left \{N_A(V) \ | \ V \subseteq A\right \}$. For $f = a_0 + a_1X_1 + \cdots + a_mX_m \in A$, we denote by $\left \{ a_0, \ldots, a_m \right \}$ or $C_f$ the set comprised of the coefficients of $f$, and for a subset $V \subseteq A$, $C_V = \bigcup_{f \in U}C_f$.

The following theorem establishes a bijective correspondence between weak annihilators of $A$ and weak annihilators of $R$. This result generalizes \cite[Theorem 3.21]{ReyesSuarez2021} which was formulated for classical annihilators.

\begin{theorem}\label{Theorem3.4}
If $A = \sigma(R)\left \langle x_1, \dotsc, x_n \right \rangle$ is a skew PBW extension over a $(\Sigma, \Delta)$-compatible and NI ring $R$, then the correspondence $\varphi :{\rm NAnn}_R(R) \rightarrow {\rm NAnn}_A(A)$, given by $\varphi(N_R(U)) = N_R(U)A$, for every $N_R(U) \in {\rm NAnn}_R(R)$, is bijective.
\end{theorem}
\begin{proof}
First, let us show that $\varphi$ is well defined, i.e., $N_A(U) =
N_R(U)A$, for every nonempty subset $U$ of $R$. Consider $f = a_0 + a_1X_1 + \cdots + a_lX_l \in N_R(U)A$ and $r\in U$. Since $a_i \in N_R(U)$, for all $0 \leq i \leq l$, this implies that $ra_i \in N(R)$, for each $i$. Thus, by Corollary \ref{MSc2.3.20}, we get $rf=ra_0+ra_1X_1 + \cdots + ra_lX_l \in N(R)A = N(A)$. Hence, $N_R(U)A \subseteq N_A(U)$. 

Now, let us prove that $N_A(U) \subseteq N_R(U)A$. If $f = a_0+a_1X_1+\cdots+a_lX_l \in N_A(U)$, then $rf = ra_0+ra_1X_1+\cdots+ra_lX_l \in N(A) = N(R)A$, for every element $r$ of $U$. This means that $ra_i \in N(R)$, for all $0 \leq i \leq l$, and so $a_i \in N_R(U)$, for each $i$, whence $f \in N_R(U)A$. Therefore, $N_A(U) = N_R(U)A$. 

Let $U,V$ be a two subsets of $R$ such that $\varphi(N_R(U))=\varphi(N_R(V))$. By definition of $\varphi$, $N_A(U)=N_A(V)$. In particular, we obtain $N_R(U)=N_R(V)$, that is, $\varphi$ is injective. 

Finally, let us show that $\varphi$ is surjective. Let $N_A(V) \in {\rm NAnn}_A(A)$ and consider $g = b_0 + b_1Y_1 + \cdots + b_sY_s \in N_A(V)$, for a subset $V$ of $A$. Then $fg \in N(A)$, for all $f = a_0 + a_1X_1 + \cdots + a_lX_l \in V$, and by Proposition \ref{ReyesSuarez2019-2Theorem 4.7}, $a_ib_j \in N(R)$, for each $i, j$. Thus, $b_j \in N_R(C_V)$, for all $0 \leq j \leq s$, whence $g \in N_R(C_V)A$, and so $N_A(V) \subseteq N_R(C_V)A$. Since $N_R(C_V)A \subseteq N_A(V)$, we obtain that $N_A(V) = N_R(C_V)A = \varphi(N_R(C_V))$, i.e., $\varphi$ is surjective.
\end{proof}

The following theorem generalizes \cite[Theorem 2.1]{OuyangBirkenmeier2012}.

\begin{theorem}\label{theorem3.4}
Let $A = \sigma(R)\left \langle x_1, \dots , x_n \right \rangle$ be a skew PBW extension over a $(\Sigma, \Delta)$-compatible and NI ring $R$. If for each subset $X \nsubseteq N(R)$, $N_R(X)$ is generated as an ideal by a nilpotent element, then for each subset $U \nsubseteq N(A)$, $N_A(U)$ is generated as an ideal by a nilpotent element.
\end{theorem}
\begin{proof}
Let $U$ be a subset of $A$ with $U \nsubseteq N(A)$. By Corollary \ref{MSc2.3.21}, $C_U \nsubseteq N(R)$. There exists $c\in N(R)$ such that $N_R(C_U)= c R$. Let us show that $N_A(U)=c A$. Let $f = a_0 + a_1X_1+ \cdots + a_lX_l \in U$ and $g= b_0 +b_1Y_1 + \cdots + b_sY_s \in A$. The idea is to show that $fcg \in N(A)$. Using that $cg= cb_0 + cb_1Y_1+ \cdots +cb_sY_s$, we get
\begin{center}
    $\displaystyle fcg = \sum_{k=0}^{s+l}\left ( \sum_{i+j=k} a_iX_icb_jY_j \right )=\sum_{k=0}^{s+l}\left ( \sum_{i+j=k} a_i\sigma^{\alpha_i}(cb_j)X_iY_j + a_ip_{\alpha_i,cb_j}Y_j \right)$.
\end{center}
Since $a_icb_j \in N(R)$, for every $0 \leq i \leq l$, $0\leq j \leq s$, \cite[Proposition 3.3]{HashemiKhalilAlhevaz2019} this implies that $a_i\sigma^{\alpha_i}(cb_j) \in N(R)$. Notice that $a_i\sigma^{\alpha}(\delta^{\beta}(cb_j))$ and $a_i\delta^{\beta}(\sigma^{\alpha}(cb_j))$ are elements of $N(R)$, for every $\alpha, \beta \in \mathbb{N}^n$. In addition, by Proposition \ref{Reyes2015Proposition2.9}, the polynomial $p_{\alpha_i,cb_j}$ involves elements obtained evaluating $\sigma$'s and $\delta$'s (depending on the coordinates of $\alpha_i$) in the element $cb_j$. Thus, $a_ip_{\alpha_i,cb_j}\in N(R)$ and $a_icb_j \in N(R)$, for every $i, j$, whence $fcg\in N(R)A = N(A)$, and therefore $cg \in N_A(U)$. 

Now, let us consider $f = a_0 + a_1X_1 +\cdots + a_lX_l \in N_A(U)$. By definition, for all $g=b_0 + b_1Y_1+ \cdots +b_sY_s$ it is satisfied that $fg \in N(A)$. From Corollary \ref{MSc2.3.21}, $fg \in N(R)A$, and so
\begin{center}
    $\displaystyle fg = \sum_{k=0}^{s+l}\left ( \sum_{i+j=k} a_iX_ib_jY_j \right )=\sum_{k=0}^{s+l}\left ( \sum_{i+j=k} a_i\sigma^{\alpha_i}(b_j)X_iY_j + p_{\alpha_i,b_j}Y_j \right) \in N(R)A$.
\end{center}
Again, since $a_i\sigma^{\alpha_i}(b_j)\in N(R)$, for every $i,j$, \cite[Proposition 3.3]{HashemiKhalilAlhevaz2019} implies that $a_ib_j \in N(R)$. Hence, $a_i \in N_R(C_U) = c R$, for each $0 \leq i \leq l$. This means that there exist elements $r_i \in R$ such that $a_i=cr_i$, for every $i$. In this way, $f = c(r_0 +r_1X_1 +\cdots + r_lX_l) \in cA$, that is, $f \in c A$. 
\end{proof}

\begin{corollary}[{\cite[Theorem 2.1]{OuyangBirkenmeier2012}}]
Let $R$ be a $(\sigma, \delta)$-compatible ring $R$. If for each subset $X\nsubseteq N(R)$, $N_R(X)$ is generated as an ideal by a nilpotent element, then for each subset $U\nsubseteq N(R[x;\sigma, \delta])$, $N_{R[x;\sigma,\delta]}(U)$ is generated as an ideal by a nilpotent element.
\end{corollary}

\begin{example}
\begin{enumerate}
\item [\rm (i)] Let $A$ be the skew PBW extension over the ring $\frac{\mathbb{F}_4[z]}{\left \langle z^2 \right \rangle}$ formulated in Example \ref{ExampleF4}. For any subset $X\subseteq \frac{\mathbb{F}_4[z]}{\left \langle z^2 \right \rangle}$ with $X \nsubseteq N\left(\frac{\mathbb{F}_4[z]}{\left \langle z^2 \right \rangle} \right)$, we obtain that $N_{\frac{\mathbb{F}_4[z]}{\left \langle z^2 \right \rangle}}(X)= z\frac{\mathbb{F}_4[z]}{\left \langle z^2 \right \rangle}$, where $z \in N\left(\frac{\mathbb{F}_4[z]}{\left \langle z^2 \right \rangle} \right)$. From Theorem \ref{theorem3.4}, we conclude that $N_{A}U=wA$, for any subset $U \subseteq A$ where $U\nsubseteq N(A)$, $w \in N(A)$ and $A=\sigma(R)\left \langle x_{1,0},x_{1,1},x_{1,2},x_{2,0},x_{2,1},x_{2,2} \right \rangle$.
\item[\rm (ii)] Consider Example \ref{ExampleMatrix}. For any subset $X\subseteq R$ with $X \nsubseteq N(R)$ it is satisfied that $N_R(X)= rR$, where $r\in N(R)$. Thus, by Theorem \ref{theorem3.4} we have that $N_{R[x;\overline{\sigma},\overline{\delta}]}U = wR[x;\overline{\sigma},\overline{\delta}]$, for any subset $U \subseteq R[x;\overline{\sigma},\overline{\delta}]$, where $U\nsubseteq N(R[x;\overline{\sigma},\overline{\delta}])$ and $w \in N(R[x;\overline{\sigma},\overline{\delta}])$. 
\end{enumerate}
\end{example}

The following theorem is a consequence of Theorem \ref{theorem3.4}  considering skew PBW extensions of the endomorphism type.  This result generalizes \cite[Theorem 2.2]{OuyangBirkenmeier2012}. 

\begin{theorem}\label{theorem3.4.1}
If $A = \sigma(R)\left \langle x_1, \dots , x_n \right \rangle$ is a skew PBW extension of endomorphism type over a $\Sigma$-compatible and NI ring $R$, then the following statements are equivalent:
\begin{enumerate}
    \item [{\rm (1)}] For each subset $X \nsubseteq N(R)$, $N_R(X)$ is generated as an ideal by a nilpotent element.
    \item[{\rm (2)}] For each subset $U \nsubseteq N(A)$, $N_A(U)$ is generated as an ideal by a nilpotent element.
\end{enumerate}
\end{theorem}
\begin{proof}
 By Theorem \ref{theorem3.4}, it suffices to show (2) $\Rightarrow$ (1). Let $X$ be a subset of $R$ with $X \nsubseteq N(R)$. Then $X \nsubseteq N(A)$, and so there exists $f = a_0 + a_1X_1 + \cdots + a_lX_l \in N(A)$ such that $N_A(X) = fA$. Notice that $f = a_0 +a_1X_1+\cdots+a_lX_l \in N(A)$, whence $a_i \in N(R)$, for all $0 \leq i \leq l$, by Corollary \ref{MSc2.3.21}. Suppose that $a_0 \neq 0$, and let us see that $N_R(X) = a_0R$. Since $a_0 \in N(R)$ and $N(R)$ is an ideal of $R$, we obtain $pa_0R \subseteq N(R)$, for each $p \in X$, whence $a_0R \subseteq N_R(X)$. If $m \in N_R(X)$, then $m \in N_A(X)$. Thus, there exists $g(x) = b_0 +b_1Y_1+\cdots+b_sY_s \in A$ such that
\begin{center}
  \begin{center}
    $\displaystyle m = fg = \sum_{k=0}^{s+l}\left ( \sum_{i+j=k} a_iX_ib_jY_j \right )=\sum_{k=0}^{s+l}\left ( \sum_{i+j=k} a_i\sigma^{\alpha_i}(b_j)X_iY_j \right )$
\end{center}  
\end{center}
Then $m = a_0b_0 \in a_0R$, and thus $N_R(X) \subseteq a_0R$, which means that $N_R(X) = a_0R$, with $a_0 \in N(R)$.
\end{proof}

\begin{corollary}[{\cite[Theorem 2.2]{OuyangBirkenmeier2012}}]
Let $R$ be a $(\sigma, \delta)$-compatible ring $R$. Then the following statements are equivalent:
\begin{enumerate}
\item [\rm (1)] For each subset $X\nsubseteq N(R)$, $N_R(X)$ is generated as an ideal by a nilpotent element. 
\item [\rm (2)]	For each subset $U\nsubseteq N(R[x;\sigma])$, $N_{R[x;\sigma]}(U)$ is generated as an ideal by a nilpotent element.
\end{enumerate}
\end{corollary}

Next, we present a theorem for weak annihilators of principal right ideals. 

\begin{theorem}\label{Theorem3.1.6}
Let $A = \sigma(R)\left \langle x_1, \dots , x_n \right \rangle$ be a skew PBW extension over a $(\Sigma, \Delta)$-compatible and NI ring $R$. If for each principal right ideal $p R \nsubseteq N(R)$, $N_R(p R)$ is generated as an ideal by a nilpotent element, then for each principal right ideal $f A \nsubseteq N(A)$, $N_A(f A)$ is generated as an ideal by a nilpotent element.
\end{theorem}

\begin{proof}
Let $f= a_0 +a_1X_1+ \cdots + a_lX_l \in A$ with $fA \nsubseteq N(A)$. If $a_i R \subseteq N(R)$, for all $0 \leq i \leq l$, by Corollary \ref{MSc2.3.21}, we have that $f A \subseteq N(A)$, which is a contradiction. In this way, there exists $0 \leq i \leq l$ such that $a_i R\nsubseteq N(R)$. Thus, there exists $c \in N(R)$ such that $N_R(a_iR) =c R$. Let us show that $N_A(fA)= c A$. Consider $g = b_0 + b_1Y_1 +\cdots + b_sY_s$ and $h=c_0 + c_1Z_1 + \cdots + c_tZ_t$ elements of $A$. By Theorem \ref{theorem3.4}, $gch \in N(R)A = N(A)$ and by Corollary \ref{MSc2.3.21}, for every $f \in A$, $fgch \in N(A) = N(R)A$, since that $N(A)$ is an ideal. Therefore, we obtain that $ch \in N_A(fA)$, that is, $cA \subseteq N_A(fA)$. Let $p = p_0 +p_1Y_1 +\cdots + p_sY_s \in N_A(f A)$. Then $f A p \subseteq N(A)$, for every $f = a_0 + a_1X_1 + \cdots + a_lX_l \in A$. In particular, $f R p \subseteq N(A)$. Let $r \in R$. We get $rp = rp_0 + rp_1Y_1 \cdots + rp_sY_s$, which implies that 
{\small{\begin{center}
    $\displaystyle frp = \sum_{k=0}^{s+l}\left ( \sum_{i+j=k} a_iX_irp_jY_j \right )=\sum_{k=0}^{s+l}\left ( \sum_{i+j=k} a_i\sigma^{\alpha_i}(rp_j)X_iY_j + a_ip_{\alpha_i,rp_j}Y_j \right ) \in N(R)A$.
\end{center}}}
Therefore, $a_i\sigma^{\alpha_i}(rp_j)\in N(R)$ and so $a_irp_j \in N(R)$, for every $i,j$. In particular, we obtain $p_j \in N_R(a_i R) = c R$, and thus there exists $r_j \in R$ such that $p_j=cr_j$  Hence, $p= p_0 +p_1Y_1+ \cdots + p_sY_s= c(r_0 + r_1Y_1 + \cdots + r_sY_s) \in c  A $. Therefore, we conclude that $p \in c A$ and so $N_A(fA)\subseteq cA$.
\end{proof}

\begin{corollary}[{\cite[Theorem 2.3]{OuyangBirkenmeier2012}}]
Let $R$ be a $(\sigma,\delta)$-compatible 2-primal ring. If for each principal right ideal $p R\nsubseteq N(R)$, $N_R(p R)$ is generated as an ideal by a nilpotent element, then for each principal right ideal $f(x) R[x;\sigma,\delta] \nsubseteq N(R[x;\sigma,\delta])$, $N_{R[x;\sigma,\delta]} (f(x) R[x;\sigma,\delta])$ is generated as an ideal by a nilpotent element.
\end{corollary}

\begin{example}
\begin{enumerate}
\item[\rm (i)] Once more again, let $A$ be the skew PBW extension over the ring $\frac{\mathbb{F}_4[z]}{\left \langle z^2 \right \rangle}$ considered in Example \ref{ExampleF4}. Let $p \in \frac{\mathbb{F}_4[z]}{\left \langle z^2 \right \rangle}$ such that $p \frac{\mathbb{F}_4[z]}{\left \langle z^2 \right \rangle} \nsubseteq N\left(\frac{\mathbb{F}_4[z]}{\left \langle z^2 \right \rangle} \right)$. Notice that $N_{\frac{\mathbb{F}_4[z]}{\left \langle z^2 \right \rangle}}\left(p\frac{\mathbb{F}_4[z]}{\left \langle z^2 \right \rangle}\right)=  z \frac{\mathbb{F}_4[z]}{\left \langle z^2 \right \rangle}  $, where $z \in N\left(\frac{\mathbb{F}_4[z]}{\left \langle z^2 \right \rangle} \right)$.  By Theorem \ref{Theorem3.1.6} $N_{A}(fA)$ is generated by a nilpotent element for any principal right ideal $fA$ with $fA \nsubseteq N(A)$ and $A=\sigma(R)\left \langle x_{1,0},x_{1,1},x_{1,2},x_{2,0},x_{2,1},x_{2,2} \right \rangle$.
\item[\rm (ii)] In Example \ref{ExampleMatrix}, let $p \in R$ such that $pR \nsubseteq N(R)$, where $p= \begin{pmatrix} p(t) & q(t)\\0 & q(t) \end{pmatrix}$, for some $p(t), q(t) \in \Bbbk[t]$ and $p(t) \neq 0$. It is straightforward to see that 
\begin{center}
    $N_R(pR)=\left \{ \begin{pmatrix} 0 & q(t)\\0 &0 \end{pmatrix} \ | \ q(t) \in \Bbbk[t]\ \right \} =  \begin{pmatrix} 0 & 1\\0 &0 \end{pmatrix}R  $, 
\end{center}
with $\begin{pmatrix} 0 & 1\\0 &0 \end{pmatrix} \in N(R)$. By Theorem \ref{Theorem3.1.6},  $N_{R[x;\overline{\sigma},\overline{\delta}]}(fR[x;\overline{\sigma},\overline{\delta}])$ is generated by a nilpotent element for any principal right ideal $fR[x;\overline{\sigma},\overline{\delta}]$ where $fR[x;\overline{\sigma},\overline{\delta}] \nsubseteq N(R[x;\overline{\sigma},\overline{\delta}])$ with $f \in R[x;\overline{\sigma},\overline{\delta}]$.
\end{enumerate}
\end{example}

\begin{theorem}\label{Theorem3.13}
If $A = \sigma(R)\left \langle x_1, \dots , x_n \right \rangle$ is a skew PBW extension of endomorphism type over a $\Sigma$-compatible and NI ring $R$, then the following statements are
equivalent:
\begin{enumerate}
    \item[{\rm (1)}] For each principal right ideal $p R \nsubseteq N(R)$, $N_R(pR)$ is generated as an ideal by a nilpotent element.
   \item[{\rm (2)}] For each principal right ideal $f A \nsubseteq N(A), N_A(f A)$ is generated as an ideal by a nilpotent element. 
\end{enumerate}
\end{theorem}
\begin{proof}
By Theorem \ref{Theorem3.1.6}, it suffices to show ${\rm (ii)} \Rightarrow {\rm (i)}$. Let $ pR$ be a principal right ideal of $R$ with $p R \nsubseteq N(R)$. In particular, we have that $pR \nsubseteq N(A)$. Therefore, there exists $f = a_0 + a_1X_1 + \cdots + a_lX_l \in N(A)$ such that $N_A(pR) =fA$. With this in mind, notice that $f = a_0 + a_1X_1 + \dotsc + a_lX_l \in N(A)$, whence $a_i \in N(R)$, for all $0 \leq i \leq l$, by Corollary \ref{MSc2.3.21}. We may assume that $a_0 \neq 0$.
Let us show that $N_R(pR) = a_0R$. Since $a_0 \in N(R)$ and $N(R)$ is an ideal of $R$, we obtain $pra_0R \subseteq N(R)$, for each $r \in R$. Thus $a_0 R \subseteq N_R(pR)$. If $m \in N_R(pR)$, then $m \in N_A(pR)$, and so there exists $g = b_0 +b_1Y_1+\cdots+b_sY_s \in A$ such that
\begin{center}
  \begin{center}
    $\displaystyle m = fg = \sum_{k=0}^{s+l}\left ( \sum_{i+j=k} a_iX_ib_jY_j \right )=\sum_{k=0}^{s+l}\left ( \sum_{i+j=k} a_i\sigma^{\alpha_i}(b_j)X_iY_j \right )$
\end{center}  
\end{center}
Hence, we have $m = a_0b_0 \in a_0R$, whence $N_R(pR) \subseteq a_0R$. Therefore, we conclude that $N_R(p R) = a_0R$ where $a_0 \in N(R)$.
\end{proof}

\begin{corollary}
	[{\cite[Theorem 2.4]{OuyangBirkenmeier2012}}] Let $R$ be a $\sigma$-compatible 2-primal ring. Then the following statements are equivalent:
	\begin{enumerate}
	\item [\rm (1)] For each principal right ideal $p R\nsubseteq N(R)$, $N_R(pR)$ is generated as an ideal by a nilpotent element.
	\item [\rm (2)] For each principal right ideal $f(x) R[x;\sigma]\nsubseteq N(R[x;\sigma])$, $N_{R[x;\sigma]}(f(x) R[x;\sigma])$ is generated as an ideal by a nilpotent element.	
	\end{enumerate}
\end{corollary}

\begin{theorem}\label{theorem3.9}
Let $A = \sigma(R)\left \langle x_1, \dots , x_n \right \rangle$ be a skew PBW extension over a $(\Sigma, \Delta)$-compatible and NI ring $R$. If for each $p \notin N(R)$, $N_R(p)$ is generated as an ideal by a nilpotent element, then for each $f \notin N(A)$, $N_A(f)$ is generated as an ideal by a nilpotent element.
\end{theorem}
\begin{proof}
Consider $f= a_0 + a_1X_1 + \cdots + a_lX_l \in A$ with $f \notin N(A)$, and let us see that $N_A(f)$ is generated as an ideal by a nilpotent element. If $a_i R \subseteq N(R)$, for all $0 \leq i \leq l$,  Corollary \ref{MSc2.3.21} implies that $f \in N(A)$, which is a contradiction. Hence, there exist $0 \leq i \leq l$ such that $a_i \notin N(R)$, and so there is $c \in N(R)$ such that $N_R(a_i) = cR$. The idea is to show that $N_A(f)= cA$. If $h = c_0 + c_1Z_1 +\cdots + c_tZ_t \in A$, by Proposition \ref{theorem3.4}, we have $fch \in N(R)A = N(A)$. Therefore, we obtain that $ch \in N_A(f)$. On the other hand, let $p = p_0 +p_1Y_1+ \cdots + p_sY_s \in N_A(f)$. For $f= a_0 +a_1X_1+ \cdots + a_lX_l \in A$, we have $fp \in N(A)$. Thus, we obtain that
\begin{center}
    $\displaystyle fp = \sum_{k=0}^{s+l}\left ( \sum_{i+j=k} a_iX_ip_jY_j \right )=\sum_{k=0}^{s+l}\left ( \sum_{i+j=k} a_i\sigma^{\alpha_i}(p_j)X_iY_j + a_ip_{\alpha_i,p_j}Y_j \right ) \in N(R)A$
\end{center}
Therefore, $a_i\sigma^{\alpha_i}(p_j)\in N(R)$ and so $a_ip_j \in N(R)$ for every $i,j$. In particular, we obtain that $p_j \in N_R(a_i) = c R$, and thus there exists $r_j \in R$ such that $p_j=cr_j$, whence $p= p_0 +p_1Y_1+ \cdots + p_sY_s= c(r_0 +r_1Y_1+ \cdots + r_sY_s) \in c A $. Therefore, we have proved that $p \in c A$. 
\end{proof}

\begin{corollary}
[{\cite[Theorem 2.5]{OuyangBirkenmeier2012}}] Let $R$ be a $(\sigma,\delta)$-compatible 2-primal ring. If for each $p\notin N(R)$, $N_R(p)$ is generated as an ideal by a nilpotent element, then for each $f(x)\notin N(R[x;\sigma,\delta])$, $N_{R[x;\sigma,\delta]}(f(x))$ is generated as an ideal by a nilpotent element.	
\end{corollary}

\begin{example}
\begin{enumerate}
\item[\rm (i)] If $A$ is the skew PBW extension over the ring $\frac{\mathbb{F}_4[z]}{\left \langle z^2 \right \rangle}$, let $p \in \frac{\mathbb{F}_4[z]}{\left \langle z^2 \right \rangle}$ such that $p  \notin N\left(\frac{\mathbb{F}_4[z]}{\left \langle z^2 \right \rangle} \right)$. Notice that $N_{\frac{\mathbb{F}_4[z]}{\left \langle z^2 \right \rangle}}\left(p\right)=  z \frac{\mathbb{F}_4[z]}{\left \langle z^2 \right \rangle}$, where $z \in N\left(\frac{\mathbb{F}_4[z]}{\left \langle z^2 \right \rangle} \right)$.  By using Theorem \ref{theorem3.9} we conclude that the ideal $N_{A}(f)$ is generated by a nilpotent element for any element $f$ where $f \notin N(A)$ and $A=\sigma(R)\left \langle x_{1,0},x_{1,1},x_{1,2},x_{2,0},x_{2,1},x_{2,2} \right \rangle$.
\item[\rm (ii)] In Example \ref{ExampleMatrix}, let $p \in R$ such that $p \notin N(R)$, where $p= \begin{pmatrix} p(t) & q(t)\\0 & p(t) \end{pmatrix}$, for some $p(t), q(t) \in \Bbbk[t]$ and $p(t) \neq 0$. It is easy to see that 
\begin{center}
    $N_R(p)=\left \{ \begin{pmatrix} 0 & q(t)\\0 &0 \end{pmatrix} \ | \ q(t) \in \Bbbk[t]\ \right \} =  \begin{pmatrix} 0 & 1\\0 &0 \end{pmatrix}R  $, 
\end{center}
where $\begin{pmatrix} 0 & 1\\0 &0 \end{pmatrix} \in N(R)$. By using Theorem \ref{theorem3.9} we obtain that the ideal $N_{R[x;\overline{\sigma},\overline{\delta}]}(f)$ is generated by a nilpotent element of $R[x;\overline{\sigma},\overline{\delta}]$, for any element $f \in R[x;\overline{\sigma},\overline{\delta}]$ where $f \notin N(R[x;\overline{\sigma},\overline{\delta}])$.
\end{enumerate}
\end{example}

For a skew PBW extension of endomorphism type over a ring $R$, the following theorem characterizes extensions for which every weak annihilator of an element is generated by a nilpotent element. Our result extends \cite[Theorem 2.6]{OuyangBirkenmeier2012}.

\begin{theorem}\label{Theorem3.18}
If $A= \sigma(R)\left \langle x_1, \dots , x_n \right \rangle$ is a skew PBW extension of endomorphism type over a $\Sigma$-compatible and NI ring $R$, then the following statements are
equivalent:
\begin{enumerate}
    \item[{\rm (1)}] For each $p \notin N(R)$, $N_R(p)$ is generated as an ideal by a nilpotent element.
 \item[{\rm (2)}] For each $f \notin N(A)$, $N_A(f)$ is generated as an ideal
by a nilpotent element. 
\end{enumerate}
\end{theorem}
\begin{proof}
By Theorem \ref{theorem3.9}, it suffices to show ${\rm (ii)} \Rightarrow {\rm (i)}$. Let $p \in R$ with $p \notin N(R)$, whence $p \notin N(A)$. Thus, there exists $f = a_0 + a_1X_1 + \cdots + a_lX_l \in N(A)$ such that $N_A(p) =fA$. Notice that $f = a_0 +a_1X_1+\cdots+a_lX_l \in N(A)$ and so $a_i \in N(R)$, for all $0 \leq i \leq l$ by Corollary \ref{MSc2.3.21}. We may assume that $a_0 \neq 0$. Now, we show that $N_R(p) = a_0 R$. Since $a_0 \in N(R)$ and $N(R)$ is an ideal of $R$, we obtain $p a_0 R \subseteq N(R)$, that is, $a_0 R \subseteq N_R(p)$. If $m \in N_R(p)$, then, we get that $m \in N_A(p)$. Thus, there exists $g = b_0 +b_1Y_1+\cdots+b_sY_s \in A$ such that
\begin{center}
  \begin{center}
    $\displaystyle m=fg = \sum_{k=0}^{s+l}\left ( \sum_{i+j=k} a_iX_ib_jY_j \right )=\sum_{k=0}^{s+l}\left ( \sum_{i+j=k} a_i\sigma^{\alpha_i}(b_j)X_iY_j \right )$
\end{center}  
\end{center}
Hence, we have $m = a_0b_0 \in a_0 R$, and thus $N_R(p) \subseteq a_0R$. Therefore, we conclude that $N_R(p) = a_0 R$ where $a_0 \in N(R)$.
\end{proof}

\section{Nilpotent Associated Prime Ideals}\label{NAP}

It is well-known that associated prime ideals are an important tool in areas such as commutative algebra and algebraic geometry concerning the primary decomposition of ideals (e.g., Eisenbud \cite[Chapter 3]{Eisenbud1995} and Eisenbud and Harris \cite[Section II.3.3]{EisenbudHarris2000}). Briefly, for a ring $R$ and a right $R$-module $N_R$, the {\em right annihilator} of $N_R$ is denoted by $r_R(N_R)$ = $\left \{r \in R\mid Nr = 0\right \}$, and $N_R$ is said to be {\em prime} if $N_R \neq 0$ and $r_R(N_R) = r_R(N_R^{'})$, for every non-zero submodule $N_R^{'} \subseteq N_R$ \cite[Definition 1.1]{Annin2004}. If $M_R$ is a right $R$-module, an ideal $P$ of $R$ is called an {\em associated prime} of $M_R$ if there exists a prime submodule $N_R \subseteq M_R$ such that $P = r_R(N_R)$. The set of associated primes of $M_R$ is denoted by ${\rm Ass}(M_R)$ \cite[Definition 1.2]{Annin2004}. One of the most important results on these ideals was proved by Brewer and Heinzer \cite{BrewerHeinzer1974} who showed that for a commutative ring $R$, the associated primes ideals of the commutative polynomial ring $R[x]$ are all extended, that is, every $P\in {\rm Ass}(R[x])$ may be expressed as $P = P_0[x]$, where $P_0 = P \cap R \in {\rm Ass}(R)$ (Faith \cite{Faith2000} presented a proof of this fact using different algebraic techniques). In several papers, Annin \cite{Annin, Annin2002, Annin2004} extended the result above to the noncommutative setting of Ore extensions, while Ni\~no et al. \cite{NinoRamirezReyes2020} formulated this result for skew PBW extensions. 

\medskip

Since Ouyang and Birkenmeier \cite{OuyangBirkenmeier2012} introduced the notion of nilpotent associated prime as a generalization of associated prime, and described all nilpotent associated primes of the skew polynomial ring $R[x;\sigma,\delta]$ in terms of the nilpotent associated primes of $R$, a natural task is to investigate this kind of ideals in the context of skew PBW extensions with the aim of generalizing the results formulated in \cite{OuyangBirkenmeier2012}. This is precisely the purpose of this section. 

\medskip

We start by recalling some definitions presented by Ouyang and Birkenmeier \cite{OuyangBirkenmeier2012}.

\begin{definition} Let $R$ be a ring. \label{definitionnilpotentprime}
\begin{enumerate}
    \item [\rm (i)] (\cite[Definition 3.1]{OuyangBirkenmeier2012}) Let $I$ be a right ideal of a non-zero ring $R$. $I$ is called a {\em right quasi-prime ideal} if $I\nsubseteq N(R)$ and $N_R(I)=N_R(I')$, for every right ideal $I' \subseteq I$ and $I'\nsubseteq N(R)$.
    \item [\rm (ii)] (\cite[Definition 3.2]{OuyangBirkenmeier2012}) Let $N(R)$ be an ideal of a ring $R$. An ideal $P$ of $R$ is called a {\em nilpotent associated prime} of $R$ if there exists a right quasi-prime ideal $I$ such that $P = N_R(I)$. The set nilpotent associated primes of $R$ is denoted by ${\rm NAss}(R)$.
\end{enumerate} 
\end{definition}

An important concept in the characterization of nilpotent associated prime ideals is the notion of good polynomial. These polynomials were first used by Shock \cite{Shock1972} with the aim of proving that uniform dimensions of a ring $R$ and the polynomial ring $R[x]$ are equal (Annin \cite{Annin, Annin2002, Annin2004} also considered such polynomials in his study of associated prime ideals of Ore extensions). A polynomial $f(x) \in R[x]$ is called \textit{good polynomial} if it has the property that the right annihilators of the coefficients of $f(x)$ are equal. Shock proved that, given any non-zero polynomial $f(x) \in R[x]$, there exists $r \in R$ such that $f(x)r$ is good \cite[Proposition 2.2]{Shock1972}. Ouyang and Birkenmeier \cite[Definition 3.3]{OuyangBirkenmeier2012} introduced the notions of nilpotent degree and nilpotent good polynomial to study the nilpotent associated prime ideals of Ore extensions. Motivated by their ideas, we present the following definitions in the setting of skew PBW extensions. 

\begin{definition}
Let $A = \sigma(R)\langle x_1,\dotsc, x_n\rangle$ be a skew PBW extension over a ring $R$, and consider an element $f = r_0 + r_1X_1 + \cdots + r_kX_k + \cdots + r_lX_l \notin N(R)A$, where $X_l \succ X_{l-1} \succ \dots \succ X_k \succ X_{k-1} \succ \cdots \succ X_1$, ${\rm lm}(f) = x^{\alpha_l}$ and ${\rm lc}(f) = r_l$.
\begin{enumerate}
    \item [\rm (i)]  If $r_k \notin N(R)$ and $r_i \in N(R)$, for all $i > k$, then we say that the {\em nilpotent degree} of $f$ is $k$, which is denoted as ${\rm Ndeg}(f)$. If $f \in N(R)A$, then we define ${\rm Ndeg}(f) =-1$.
    \item [\rm (ii)] Suppose that the nilpotent degree of $f$ is $k$. If $N_R(r_k) \subseteq N_R(r_i)$, for all $i \leq k$, then we say that $f$ is a {\em nilpotent good polynomial}.
\end{enumerate}
\end{definition}

The following result generalizes \cite[Lemma 3.1]{OuyangBirkenmeier2012} and \cite[Lemma 2.16]{OuyangLiu2012}.

\begin{theorem}\label{proposition3.4}
Let $A = \sigma(R)\langle x_1,\dotsc, x_n\rangle$ be a skew PBW extension over a $(\Sigma, \Delta)$-compatible and NI ring $R$. For any $ f = r_0 + r_1X_1 + \cdots + r_kX_k + \cdots + r_lX_l \notin N(R)A$, there exists $r \in R$ such that $fr$ is a nilpotent good polynomial.
\end{theorem}
\begin{proof}
Let us assume that the result is false and suppose that $f = r_0 + r_1X_1 + \dotsb + r_kX_k + \cdots +r_lX_l \notin N(A)$ is a counterexample of minimal nilpotent degree ${\rm Ndeg}(f) = k$, that is, $fr$ is not a nilpotent good polynomial. In particular, if $r=1$, we have that $f$ is not a nilpotent good polynomial. Hence, there exists $i < k$ such that $N_R(r_k) \nsubseteq N_r(r_i)$. Thus, we can find $b \in R$ such that $b\in N_R(r_k)$ and $b \notin N_R(r_i)$, whence $r_ib \notin N(R)$ and $r_kb \in N(R)$.

Notice that the degree $k$ coefficient of $fb$ is $r_k\sigma^{\alpha_k}(b) + \sum_{i=k+1}^l r_i p_{\alpha_i,b}$. Now, $r_k \sigma^{\alpha_k}(b) \in N(R)$ due to the $(\Sigma, \Delta)$-compatibility of $R$. On the other hand, we have ${\rm Ndeg}(f) = k$, whence $r_i \in N(R)$, for all $i > k$. This means that $r_ip_{\alpha_i,b} \in N(R)$, for all $i > k$ and so $\sum_{i=k+1}^l r_i p_{\alpha_i,b} \in N(R)$ since $N(R)$ is an ideal of $R$. Therefore, $fb$ has nilpotent degree at most $k-1$. Since $r_ib \notin N(R)$, we have $fb \notin N(R)A$. By the minimality of $k$, there exists $c \in R$ with $fbc$ a nilpotent good polynomial. However, this contradicts the fact that $f$ is a counterexample, since $bc \in R$ and $f(bc)$ is a nilpotent good polynomial. This concludes the proof. 
\end{proof}

The following theorem characterizes the nilpotent associated primes ideals in a skew PBW extension over a $(\Sigma, \Delta)$-compatible and NI ring. Later, we state the different results from the literature that are corollaries of Theorem \ref{Nilpotentprimes}.

\begin{theorem}\label{Nilpotentprimes}
If $A$ is a skew PBW extension over a $(\Sigma, \Delta)$-compatible and NI ring $R$, then
\[
{\rm NAss}(A) = \{PA\mid P\in {\rm NAss}(R)\}.
\]
\end{theorem}
\begin{proof} With the aim of establishing the desired equality, we proof the two implications.
Let $P \in {\rm NAss}(R)$. By Definition \ref{definitionnilpotentprime} (ii), there exists a right ideal $I \nsubseteq N(R)$ with $I$ a right quasi-prime ideal of $R$ and $P = N_R(I)$. Let us show that $PA = N_A(IA)$ and also that $IA$ is a quasi-prime ideal. We show first that $PA = N_A(IA)$. Let $i = a_0 +a_1X_1+ \cdots + a_mX_m \in IA$ and let $f = b_0 +b_1Y_1+ \cdots + b_lY_l \in PA $. Then,
\begin{center}
    $\displaystyle if = \sum_{k=0}^{m+l}\left ( \sum_{i+j=k}a_iX_ib_jY_j\right )= \sum_{k=0}^{m+l}\left ( \sum_{i+j=k}a_i\sigma^{\alpha_i}(b_j) X_iY_j + a_ip_{\alpha_i,b_j}Y_j\right )$
\end{center}
Since $a_ib_j \in N(R)$ for every $i,j$, Proposition \ref{PropertynilR} implies that $a_i\sigma^{\alpha_i}(b_j) \in N(R)$, Furthermore, the polynomial $p_{\alpha_i,b_j}$ involves elements obtained evaluating $\sigma$'s and $\delta$'s (depending on the coordinates of $\alpha_i$) in the element $b_j$ by Proposition \ref{Reyes2015Proposition2.9}. This means that $a_ip_{\alpha_i,b_j}\in N(R)$ for every $i,j$ and thus, we obtain $if\in N(R)A = N(A)$. Hence, $PA \subseteq N_A(IA)$. 

Conversely, let $f = b_0 +b_1Y_1+ \cdots+ b_mY_m \in N_A(IA)$, then $if \in N(A) = N(R)A$, for all $i = a_0 +a_1X_1+ \cdots + a_lX_l \in IA$, whence $a_i\sigma^{\alpha_i}(b_j) \in N(R)$, which implies that $a_ib_j \in N(R)$, by Proposition \ref{PropertynilR}. Since $a_i \in I$, we have  $b_j \in N_R(I) = P$. Therefore, we get $f \in PA $ and thus $N_A(IA) \subseteq PA $. Hence, we conclude that $PA = N_A(I A)$.

Since the ideal $I$ is a right quasi-prime ideal, we have $I \nsubseteq N(R)$, which implies that $IA \nsubseteq N(A)$. Let us show that for any right ideal $U$ of $R$ it is satisfied that, if $U \nsubseteq N(A)$ and $U \subseteq IA$, then $N_A(U)=N_A(IA)$. Let us see first that $N_A(I A) \subseteq N_A(U)$. If $f= a_0 +a_1X_1+ \cdots + a_mX_m \in N_A(IA)$, this means that $if \in N(A)$, for all $i \in I A$. In particular, since $U \subseteq IA$, then  $if \in N(A)$, for all $i \in U$, whence $f \in N_A(U)$. Hence, we have $N_A(I A) \subseteq N_A(U)$.

Conversely, let $C_U \subseteq R$ consisting of all coefficients of elements of $U$. Let us first consider $P'$ the right ideal of $R$ generated by $C_U$. Since $U \nsubseteq N(A) = N(R) A$, this means that $C_U \nsubseteq N(R)$, and hence $P' \subseteq I$ and $P'\nsubseteq N(R)$. Thus, as $I$ is a right quasi-prime ideal this implies that $N_R(P')=N_R(I) = P$. Now, if $f = a_0+ a_1X_1+ \cdots +a_mX_m \in N_A(U)$ and $u = u_0 + u_1Y_1 + \cdots +u_tY_t \in U$, then $uf \in N(A)$, whence $u_i\sigma^{\beta_{j}} (a_j) \in N(R)$ and therefore $u_ia_j \in N(R)$, for all $0 \leq i \leq t$, $0 \leq j \leq m$. Since $R$ is a NI ring, $N(R)$ is an ideal, $u_ia_j \in N(R)$ implies $a_ju_i \in N(R)$ and thus $u_iRa_j \in N(R)$ gives that $(u_iRa_j)^2 \in N(R)$. Therefore, 
\[
a_j \in N_R(P') = N_R(I) = P, \quad \text{for all} \quad 0 \leq j \leq m.
\]

Let $i = b_0 + b_1Z_1+ \cdots + b_rZ_r \in IA$. We have $b_ma_j \in N(R)$, and so $b_m\sigma^{\alpha}(\delta^{\beta}(a_j))$ and $a_m\delta^{\beta}(\sigma^{\alpha}(a_j))$ are elements of $N(R)$, for every $\alpha, \beta \in \mathbb{N}^n$. Therefore, we get that $if \in N(R)A = N(A)$, which implies that $f \in N_A(IA)$. Hence, we conclude $N_A(U) \subseteq N_A(IA)$. Thus, we have proved that $PA = N_A(IA)$ and also that $IA$ is a quasi-prime ideal.

Let $I \in {\rm NAss}(A)$. By Definition \ref{definitionnilpotentprime} (ii), there exists a right ideal $J \nsubseteq N(A)$ with $J$ a right quasi-prime ideal of $A$ and $I = N_A(J)$. Let $m = m_0 +m_1X_1+\cdots+m_kX_k
 +\cdots+m_nX_n \notin N(A) = N(R)A$ and $m \in J$. Since $J \nsubseteq N(A)$, we may assume that $m$ is nilpotent good and ${\rm Ndeg}(m) = k$, by Theorem \ref{proposition3.4}. We consider $J_0 = m A$ the principal right ideal of $A$ generated by $m$. Since $m \notin N(A) = N(R) A$ this implies that  $J_0 = m A \nsubseteq N(R)A = N(A)$, whence $N_A(J) = N_A(J_0) = N_A(m A) = I$ because $J$ is quasi-prime ideal. Now, we consider the right ideal $m_k R$, and let us denote $U = N_R(m_k R)$.
 
Let us prove first that $I = UA$. Let $g = b_0 +b_1Y_1+\cdots+b_lY_l \in UA$. Since $b_j \in U$, then $m_kRb_j \in N(R)$ for all $0 \leq j \leq l$. Furthermore, $m$ is nilpotent good polynomial and ${\rm Ndeg}(m) = k$, hence $m_iRb_j \in N(R)$, for all $0 \leq i \leq k$, and $0 \leq j \leq l$. On the other hand, for all $i > k$, $m_i \in N(R)$. Thus, we get $m_iRb_j \in N(R)$, for all $0 \leq i \leq n$ and $0 \leq j \leq l$. Now, for any element $h \in A$ where $h = h_0 +h_1Z_1+ \cdots+ h_pZ_p $, it is satisfied that $m_ih_db_j \in N(R)$, for all $0 \leq i \leq n$, $0 \leq d \leq p$ and $0 \leq j \leq l$. Therefore, by $(\Sigma,\Delta)$-compatibility of $R$, we obtain $mhg \in N(A)$. This implies that $g \in N_A(m A) = I$, and so $UA \subseteq I$. Conversely, let $g = b_0 +b_1Y_1+\cdots+b_lY_l \in I$. Since $mRg \in N(A)$, it follows that $m_iRb_j \in N(R)$, for all $0 \leq i \leq n$, and $0 \leq j \leq l$. Thus, we have that $b_j \in N_R(m_k R)$, for all $0 \leq j \leq l$, and so $g \in UA$. Hence, we conclude $I \subseteq UA$ which implies that $I = UA$.

Now, let us show that $m_kR$ is a quasi-prime ideal. Let $m_k R$ the principal right ideal of $A$ generated by $m_k$. Since $m_k \notin N(R)$, we have $m_k R \nsubseteq N(R)$. The idea is to show that for any right ideal $Q\subseteq m_k R$ with $Q\nsubseteq N(R)$, it follows that $N_R(Q) = N_R(m_k R)$. Assume that a right ideal $Q \subseteq m_k R$, and $Q \nsubseteq N(R)$. Then $N_R(m_k R) \subseteq N_R(Q)$ by Proposition \ref{Proposition2.1ouyang}. Now, we show that $N_R(Q) \subseteq N_R(m_k R)$. Let $W$ be the following set $W = \left \{mr \ | \ r \in Q\right \}$, and let $WA$ the right ideal of $A$ generated by $W$.

First, notice that $WA \subseteq m A$. Since $Q \nsubseteq N(R)$, there exists $a \in R$ such that $m_ka \in Q$ and $m_ka \notin N(R)$. If $m_k(m_ka) \in N(R)$, then we have $m_ka \in N(R)$ which contradicts to the fact that $m_ka \notin N(R)$. Thus $m_k(m_ka) \notin N(R)$, and hence $m(m_ka) \notin N(A)$, by Proposition \ref{ReyesSuarez2019-2Theorem 4.7}. This implies that $WA \nsubseteq N(A)$. Since $J$ is quasi-prime ideal, we obtain $N_A(WA) = N_A(m A) = I$.

Suppose $q \in N_R(Q)$. Then $rq \in N(R)$, for each $r \in Q$. Now, for any $mrf \in WA$ where $f = a_0 +a_1Y_1+\cdots+a_lY_l \in A$, the term of $mr f$ is $m_iX_ira_jY_j$. The idea is to show that $m_iX_ira_jY_j \in N(R)$. Since $rq \in N(R)$ and $N(R)$ is an ideal, it follows that
\begin{center}
$\displaystyle rq \in N(R) \Rightarrow qr \in N(R) \Rightarrow ra_j(qr)a_jq \in N(R) \Rightarrow ra_jq \in N(R)$.    
\end{center}
If $ra_jq \in N(R)$, then  $m_ira_jq \in N(R)$. Thus, due to the $(\Sigma, \Delta)$-compatibility of $R$, we have that $m_iX_ira_jY_jq \in N(R)A$ which implies that $mrfq \in N(R)A = N(A)$. Hence, for any $\sum mr_if_i \in WA$ it follows that $\sum (mr_if_i)q \in N(A)$. Therefore, $q \in N_A(WA) = I = UA$, and so $q \in U = N_R(m_k R)$. Thus, $N_R(Q) \subseteq N_R(m_k R)$, and this implies that $N_R(Q) = N_R(m_k R)$. Hence, we conclude that $m_k R$ is quasi-prime ideal.
\end{proof}
\begin{corollary}[{\cite[Theorem 3.1]{OuyangBirkenmeier2012}}]
Let $R$ be a $(\sigma,\delta)$-compatible 2-primal ring. Then
\[
{\rm NAss}(R[x;\sigma,\delta]) = \{P[x;\sigma,\delta]\mid P\in {\rm NAss}(R)\}.
\]
\end{corollary}

\begin{corollary} Let $R$ be a ring.
\begin{enumerate}
    \item[\rm (1)] (\cite[Theorem 3.1]{OuyangLiu2012}) Let $R$ be a $\delta$-compatible reversible ring. Then
    \[ {\rm NAss}(R[x;\delta]) = \{P[x; \delta]\mid P\in {\rm NAss}(R)\}.
    \]
    \item[\rm (2)] (\cite[Corollary 3.2]{OuyangLiu2012}) Let $R$ be a reversible ring. Then
    \[{\rm NAss}(R[x]) = \{P[x]\mid P\in {\rm NAss}(R)\}.
    \]
\end{enumerate}
\end{corollary}

\begin{corollary} Let $R$ be a ring and consider $M_R=R_R$. 
\begin{enumerate}
    \item [\rm (1)](\cite[Theorem 3.1.2]{NinoRamirezReyes2020}) Let $A = \sigma(R)\langle x_1,\dotsc, x_n\rangle$ be a skew PBW extension over $R$ and $M_R$ a right $R$-module. If $M_R$ is $(\Sigma, \Delta)$-compatible, then
\[
{\rm Ass}(M\langle X\rangle_A) = \{P\left \langle X \right \rangle\mid P\in {\rm 
Ass}(R)\}.
\]
    \item [\rm (2)] ({\cite[Theorem 2.1]{Annin2004}})  Let $S=R[x;\sigma,\delta]$ and $M_R$ a right $R$-module. If $M_R$ is $(\sigma,\delta)$-compatible, then
    \[
{\rm Ass}(M[X]_S) = \{P[X]\mid P \in {\rm 
Ass}(M_R)\}.
\]
    \item [\rm (3)] Let $S=R[x;\delta]$ be a differential polynomial ring of $R$ and $M_R$ a right $R$-module. If $M_R$ is $\delta$-compatible, then
    \[
{\rm Ass}(M[X]_S) = \{P[X]\mid {P} \in {\rm 
Ass}(M_R)\}.
\]
    
\end{enumerate}
\end{corollary}

\begin{example} 
Since Example \ref{ExampleMatrix} satisfies the conditions of Theorem \ref{Nilpotentprimes}, it follows that ${\rm NAss}(R[x;\overline{\sigma},\overline{\delta}])={\rm NAss}(R)[x;\overline{\sigma},\overline{\delta}]$. On the other hand, the right ideals of $R$ are given by 
\begin{center}
    $I_1=\begin{pmatrix} 0 & 0 \\ 0 & 0 \end{pmatrix}R,\ \ \ \ \ \ \ \ \ I_2=\begin{pmatrix} 0 & 0 \\ 0 & q(t) \end{pmatrix}R, \ \ \ \ \ \ \ \ \ I_3=\begin{pmatrix} 0 & p(t) \\ 0 & 0 \end{pmatrix}R$,
\end{center}

\begin{center}
     $I_4=\begin{pmatrix} 0 & p(t) \\ 0 & q(t) \end{pmatrix}R,\ \ \ \ \ \   I_5=\begin{pmatrix} q(t) & p(t) \\ 0 & 0 \end{pmatrix}R,\ \ \ \ \ \   I_6=\begin{pmatrix} q(t) & p(t) \\ 0 & s(t) \end{pmatrix}R$,
\end{center}
for $p(t),q(t) \in \Bbbk[t]$. Additionally, we can observe that the only quasi-prime ideals of $R$ are $I_2,I_4,I_5$, and $I_6$. Furthermore,  $N_R(I_2)=N_R(I_4)=N_R(I_5)=N_R(I_6)=I_3$, where $I_3 = N(R)$. Thus, it follows that ${\rm NAss}(R)=\left \{ N(R)\right \}=\left \{I_3\right \}$. Therefore, we conclude that ${\rm NAss}(R[x;\overline{\sigma},\overline{\delta}]) = N(R)[x;\overline{\sigma},\overline{\delta}]=I_3[x;\overline{\sigma},\overline{\delta}]$.
\end{example}

\section{Examples}\label{examplespaper}
As is clear, the relevance of the results presented in the paper is appreciated when we apply them to algebraic structures more general than those considered by Ouyang et al. \cite{OuyangBirkenmeier2012, OuyangLiu2012}, i.e., noncommutative rings which cannot be expressed as Ore extensions. With the aim of presenting different examples of this kind of rings, in this section we consider several and remarkable families of rings that have been studied in the literature which are subfamilies of skew PBW extensions. Of course, the list of examples is not exhaustive.

\begin{definition}[{\cite{BellSmith1990}; \cite[Definition C4.3]{Rosenberg1995}}]\label{monito}
A $3$-{\em dimensional algebra} $\mathcal{A}$ is a $\Bbbk$-algebra generated by the indeterminates $x, y, z$ subject to the relations $yz-\alpha zy=\lambda$, $zx-\beta xz=\mu$, and $xy-\gamma yx=\nu$, where $\lambda,\mu,\nu \in \Bbbk x+\Bbbk y+\Bbbk z+\Bbbk$, and $\alpha, \beta, \gamma \in \Bbbk^{*}$. $\mathcal{A}$ is called a \textit{3-dimensional skew polynomial $\Bbbk$-algebra} if the set $\{x^iy^jz^k\mid i,j,k\geq 0\}$ forms a $\Bbbk$-basis of the algebra.
\end{definition}

Different authors have studied ring-theoretical and geometrical properties of 3-dimensional skew polynomial algebras (e.g., \cite{Jordan1993, JordanWells1996, RedmanPhD1996, Redman1999, SarmientoReyes2022}. Next, we recall the classification of these algebras.

\begin{proposition}[{\cite{Rosenberg1995}, Theorem C4.3.1]}]\label{quasipolynomial}
 Up to isomorphism, a 3-dimensional skew polynomial $\Bbbk$-algebra $\mathcal{A}$ is given by the following relations:
\begin{enumerate}
    \item [\rm (1)] If $|\{\alpha,\beta,\gamma\}|=3$, then $\mathcal{A}$ is defined by the relations $yz-\alpha zy=0$, $zx-\beta xz=0$, and $xy-\gamma yx=0$.
    \item [\rm (2)] If $|\{\alpha,\beta,\gamma\}|=2$ and $\beta\not=\alpha=\gamma=1$, then $\mathcal{A}$ is defined by one of the following rules:
    \begin{enumerate}
        \item [\rm (a)] $yz-zy=z$, $zx-\beta xz=y$, and $xy-yx=x$.
        \item [\rm (b)] $yz-zy=z$, $zx-\beta xz=b$, and $xy-yx=x$.
        \item [\rm (c)] $yz-zy=0$, $zx-\beta xz=y$, and $xy-yx=0$.
        \item [\rm (d)] $yz-zy=0$, $zx-\beta xz=b$, and $xy-yx=0$.
        \item [\rm (e)] $yz-zy=az$, $zx-\beta xz=0$, and $xy-yx=x$.
        \item [\rm (f)] $yz-zy=z$, $zx-\beta xz=0$, and $xy-yx=0$.
    \end{enumerate}
    Here $a,b\in \Bbbk$ are arbitrary; all non-zero values of $b$ yield isomorphic algebras.
    \item [\rm (3)] If $\alpha=\beta=\gamma\not=1$, and if $\beta\not=\alpha=\gamma\not= 1$, then $\mathcal{A}$ is one of the following algebras:
    \begin{enumerate}
        \item [\rm (a)] $yz-\alpha zy=0$, $zx-\beta xz=y+b$, and $xy-\alpha yx=0$.
        \item [\rm (b)] $yz-\alpha zy=0$, $zx-\beta xz=b$, and $xy-\alpha yx=0$.
    \end{enumerate}
    Here $a,b\in \Bbbk$ is arbitrary; all non-zero values of $b$ yield isomorphic algebras.
    \item [\rm (4)] If $\alpha=\beta=\gamma\not=1$, then $\mathcal{A}$ is determined by the relations $yz-\alpha zy= a_1 x + b_1$, $zx-\alpha xz= a_2 y + b_2$, and $xy-\alpha yx= a_3 z + b_3$.

    If $a_i=0$, then all non-zero values of $b_i$ yield isomorphic algebras.
    \item [\rm (5)] If $\alpha=\beta=\gamma=1$, then $\mathcal{A}$ is isomorphic to one of the following algebras:
    \begin{enumerate}
        \item [\rm (a)] $yz-zy=x$, $zx-xz=y$, and $xy-yx=z$.
        \item [\rm (b)] $yz-zy=0$, $zx-xz=0$, and $xy-yx=z$.
        \item [\rm (c)] $yz-zy=0$, $zx-xz=0$, and $xy-yx=b$.
        \item [\rm (d)] $yz-zy=-y$, $zx-xz=x+y$, and $xy-yx=0$.
        \item [\rm (e)] $yz-zy=az$, $zx-xz=x$, and $xy-yx=0$.
    \end{enumerate}
    Here $a,b\in \Bbbk$ are arbitrary; all non-zero values of $b$ yield isomorphic algebras.
\end{enumerate}
\end{proposition}

It follows from Definition \ref{monito} that every $3$-dimensional skew polynomial algebra is a skew PBW extension over the field $\Bbbk$, that is, $\mathcal{A} \cong \sigma(\Bbbk)\left \langle x,y,z \right \rangle$. For these algebras, Theorem \ref{proposition3.4} guarantees the existence of nilpotent good polynomials and Theorem \ref{Nilpotentprimes} characterizes the nilpotent associated prime ideals. In particular, since $\Bbbk$ is a field (hence reduced), the only nilpotent associated prime ideal is the zero ideal.

     \begin{example}
     Following Havli\v{c}ek et al. \cite[p. 79]{HavlicekKlimykPosta2000}, the algebra $U_q'(\mathfrak{so}_3)$ is generated over a field $\Bbbk$ by the indeterminates $I_1, I_2$, and $I_3$, subject to the relations given by
    \begin{equation*}
        I_2I_1 - qI_1I_2 = -q^{\frac{1}{2}}I_3,\quad 
        I_3I_1 - q^{-1}I_1I_3 = q^{-\frac{1}{2}}I_2,\quad 
        I_3I_2 - qI_2I_3 = -q^{\frac{1}{2}}I_1,
    \end{equation*}
    where $q\in \Bbbk\ \backslash\ \{0, \pm 1\}$. It is straightforward to show that $U_q'(\mathfrak{so}_3)$ cannot be expressed as an iterated Ore extension. However, this algebra is a skew PBW extension over $\Bbbk$, i.e., $U_q'(\mathfrak{so}_3) \cong \sigma(\Bbbk) \langle I_1, I_2, I_3\rangle$ \cite[Example 1.3.3]{LFGRSV}. On the other hand, notice that $\mathbb{R}$ is a $(\Sigma,\Delta)$-compatible NI ring with the identity endomorphism of $\mathbb{R}$ and the trivial $\sigma$-derivation. Thus, since that $r_{\Bbbk}(X)=N_{\Bbbk}(X)=\left \{ 0 \right \}$, for all $X\nsubseteq N(\Bbbk)$, Theorem \ref{theorem3.4} guarantees that $N_{U_q'(\mathfrak{so}_3)}(U)$ is an ideal generated by a nilpotent element of $U_q'(\mathfrak{so}_3)$, for each subset $U \nsubseteq N(U_q'(\mathfrak{so}_3))$. Following a similar argument, by Theorem \ref{Theorem3.1.6} $N_{U_q'(\mathfrak{so}_3)}(f\ U_q'(\mathfrak{so}_3) )$ is an ideal generated by a nilpotent element, for each principal right ideal $f\ U_q'(\mathfrak{so}_3) \nsubseteq N(U_q'(\mathfrak{so}_3))$. Now, Theorem \ref{theorem3.9} guarantees that $N_{U_q'(\mathfrak{so}_3)}(f)$ is an ideal generated by a nilpotent element, for each $f \notin N(U_q'(\mathfrak{so}_3))$. Thinking about nilpotent good polynomials and nilpotent associated primes ideals, Theorem \ref{proposition3.4} allows us to conclude that the algebra $U_q'(\mathfrak{so}_3)$ has nilpotent good polynomials, while Theorem \ref{Nilpotentprimes} describes the nilpotent associated primes ideals for these algebras.

    \end{example}
    
    \begin{example}
    Zhedanov \cite[Section 1]{Zhedanov} introduced the {\em Askey-Wilson algebra} $AW(3)$ as the algebra generated by three operators $K_0, K_1$, and $K_2$, that satisfy the commutation relations
    \begin{align}
        [K_0, K_1]_{\omega} &\ = K_2, \label{Zhedanov1991relation(1.1a)} \\
        [K_2, K_0]_{\omega} &\ = BK_0 + C_1K_1 + D_1, \label{Zhedanov1991relation(1.1b)}\\ 
        [K_1, K_2]_{\omega} &\ = BK_1 + C_0K_0 + D_0, \label{Zhedanov1991relation(1.1c)}
    \end{align}
    where $B, C_0, C_1, D_0, D_1$, are the structure constants of the algebra, which Zhedanov assumes are real, and the $q$-commutator $[ - , -]_{\omega}$ is given by $[\square, \triangle]_{\omega}:= e^{\omega}\square \triangle - e^{- \omega}\triangle \square$, where $\omega$ is an arbitrary real parameter. 
    
    Notice that in the limit $\omega \to 0$, the algebra AW(3) becomes an ordinary Lie algebra with three generators ($D_0$ and $D_1$ are included among the structure constants of the algebra in order to take into account algebras of Heisenberg-Weyl type).
    
    The relations (\ref{Zhedanov1991relation(1.1a)}) - (\ref{Zhedanov1991relation(1.1c)}) can be written as
    \begin{align}
        e^{\omega}K_0K_1 - e^{-\omega}K_1K_0 &\ = K_2\\
        e^{\omega} K_2K_0 - e^{-\omega}K_0 K_2 &\ = BK_0 + C_1K_1 + D_1\\
        e^{\omega}K_1K_2 - e^{-\omega}K_2K_1 &\ = BK_1 + C_0K_0 + D_0.
    \end{align}
    According to the relations that define the algebra, it is clear that AW(3) cannot be expressed as an iterated Ore extension. Nevertheless, using techniques such as those presented in \cite[Theorem 1.3.1]{LFGRSV}, it can be shown that ${\rm AW(3)}$ is a skew PBW extension of endomorphism type, that is,  ${\rm AW(3)} \cong \sigma(\mathbb{R})\langle K_0, K_1, K_2\rangle$. Since that $\mathbb{R}$ is reduced, $r_{\mathbb{R}}(X)=N_{\mathbb{R}}(X)=\left \{ 0 \right \}$, for all $X\nsubseteq N(\mathbb{R})$. Hence, the characterization the weak annihilators of $N_{{\rm AW(3)}}(U)$, $N_{{\rm AW(3)}}(f\ {\rm AW(3)})$, and $N_{{\rm AW(3)}}(f)$,  for certain subsets, principal right ideals, and elements of algebra ${\rm AW(3)}$, follow from Theorems \ref{theorem3.4}, \ref{Theorem3.1.6}, and \ref{theorem3.9}, respectively. Furthermore, Theorems \ref{theorem3.4.1}, \ref{Theorem3.13} and \ref{Theorem3.18} are also valid for the description of the weak annihilators mentioned above. For AW(3), Theorem \ref{proposition3.4} guarantees the existence of nilpotent good polynomials, and Theorem \ref{Nilpotentprimes} describes the nilpotent associated prime ideals. In particular, since that $\mathbb{R}$ is a field (hence reduced), the only nilpotent associated prime ideal is the zero ideal. All the theorems mentioned above are valid for ${\rm AW}(3)$ if we change $\mathbb{R}$ to any reduced ring $R$. 
    
    \end{example}
    
    \begin{example}
    As is well-known, algebras whose generators satisfy quadratic relations such as Clifford algebras, Weyl-Heisenberg algebras, and Sklyanin algebras, play an important role in analysis and mathematical physics. Motivated by these facts, Golovashkin and Maximov \cite{GolovashkinMaximov2005} considered the algebras $Q(a, b, c)$, with two generators $x$ and $y$, generated by the quadratic relations 
    \begin{equation}\label{GolovashkinMaximov2005(1)}
        yx = ax^2 + bxy + cy^2,
    \end{equation}
    where the coefficients $a, b$, and $c$ belong to an arbitrary field $\Bbbk$ of characteristic zero. They studied conditions on these elements under which such an algebra has a PBW basis of the form $\{x^my^n\mid m, n \in \mathbb{N}\}$.  
    
    In \cite[Section 3]{GolovashkinMaximov2005}, Golovashkin and Maximov presented a necessary and sufficient condition for a PBW basis as above exists in the case $ac + b \neq 0$, while in \cite[Section 5]{GolovashkinMaximov2005}, for the case $ac + b = 0$, they proved that if $b \neq 0, -1$, then $Q(a, b, c)$ has a PBW basis, and if $b = -1$, then the elements $\{x^my^n\mid m, n \in \mathbb{N}\}$ are linearly independent but do not form a PBW basis of $Q(a, b, c)$.
    
    From relation (\ref{GolovashkinMaximov2005(1)}) one can see that if $a, b,$ and $c$ are not zero simultaneously, then $Q(a, b, c)$ is neither an Ore extension of $\Bbbk$, $\Bbbk[x]$ nor of $\Bbbk[y]$. On the other hand, if $b\neq 0$ and $c=0$, then one can check that $Q(a, b, c)$ is a skew PBW extension over $\Bbbk[x]$, i.e., $Q(a, b, c)\cong \sigma(\Bbbk[x])\langle y\rangle$. Note that $\Bbbk[x]$ is a $(\Sigma,\Delta)$-compatible reduced ring with the identity endomorphism of $\Bbbk[x]$ and the trivial $\sigma$-derivation. Since that $\Bbbk[x]$ is a reduced ring, we get $r_{\Bbbk[x]}(X)=N_{\Bbbk[x]}(X)=\left \{ 0 \right \}$, for all $X\nsubseteq N(\Bbbk[x])$. Hence, the characterization the weak annihilators of $N_{Q(a, b, c)}(U)$, $N_{Q(a, b, c)}(f\ Q(a, b, c))$, and $N_{Q(a, b, c)}(f)$,  for some subsets, principal right ideals, and elements of algebra $Q(a, b, c)$, follow from Theorems \ref{theorem3.4}, \ref{Theorem3.1.6}, and \ref{theorem3.9}, respectively. Also, Theorems \ref{theorem3.4.1}, \ref{Theorem3.13} and \ref{Theorem3.18} are also valid for the description of these weak annihilators. Lastly, for the algebra $Q(a,b,c)$, the existence of nilpotent good polynomials and characterization of the nilpotent associated prime ideals follow from Theorem \ref{proposition3.4} and \ref{Nilpotentprimes}, respectively. As an observation, if $a=0$, then one can check that $Q(a, b, c)$ is a skew PBW extension on $\Bbbk[y]$, i.e., $Q(a, b, c)\cong \sigma(\Bbbk[y])\langle x\rangle$ and all the previous results are satisfied too. Similarly, if we change $\Bbbk$ to a reduced ring $R$, all our observations and results are still valid for $Q(a,b,c)$. 
    \end{example}
    
    \begin{example}
    \cite[Section 25.2]{BurdikNavratil2009} The basis of algebra $\mathfrak{g} = \mathfrak{so}(5, \mathbb{C})$ consists of the elements $\bf{J}_{\alpha \beta} = -{\bf J}_{\beta \alpha}$, $\alpha, \beta = 1, 2, 3, 4, 5$ satisfying the commutation relations
    \begin{equation}\label{2009LieAlgebraso(5)(25.2)}
        [\bf{J}_{\alpha \beta}, \bf{J}_{\mu \nu}] = \delta_{\beta \mu} \bf{J}_{\alpha \nu} + \delta_{\alpha \nu}\bf{J}_{\beta \mu} - \delta_{\beta \nu}\bf{J}_{\alpha \mu} - \delta_{\alpha \mu}\bf{J}_{\beta \nu}.
    \end{equation}
    If we consider the elements
    \begin{align*}
        {\bf H}_1 & = \bf{iJ}_{12}, & \bf{H}_2 & = \bf{iJ}_{34},\\
        \bf{E}_1 & = \frac{1}{\sqrt{2}} (\bf{J}_{45} + i\bf{J}_{35}), & \bf{E}_2 & = \frac{1}{2} (\bf{J}_{23} + i\bf{J}_{13} - \bf{J}_{14} + i \bf{J}_{24}),\\
        \bf{E}_3 & = \frac{1}{\sqrt{2}} (\bf{J}_{15} - i\bf{J}_{25}), & \bf{E}_4 & = \frac{1}{2} (\bf{J}_{23} + i\bf{J}_{13} + \bf{J}_{14} - i\bf{J}_{24}), \\ 
        \bf{F}_1 & = \frac{1}{\sqrt{2}} (-\bf{J}_{45} + i\bf{J}_{35}), & \bf{F}_2 & = \frac{1}{2} (-\bf{J}_{23} + i \bf{J}_{13} + \bf{J}_{14} + i\bf{J}_{24}),\\
        \bf{F}_3 & = \frac{1}{\sqrt{2}} (-\bf{J}_{15} -i\bf{J}_{25}), & \bf{F}_4 & = \frac{1}{2}(-\bf{J}_{23} + i\bf{J}_{13} - \bf{J}_{14} - i\bf{J}_{24}),
    \end{align*}
    the commutation relations (\ref{2009LieAlgebraso(5)(25.2)}) can be expressed as
    \begin{align*}
        [\bf{H}_1, \bf{E}_1] & = 0, & [\bf{H}_1, \bf{E}_2] & = \bf{E}_2, & [\bf{H}_1, \bf{E}_3] & = \bf{E}_3, & [\bf{H}_1, \bf{E}_4] & = \bf{E}_4,\\ 
        [\bf{H}_2, \bf{E}_1] & = \bf{E}_1, & [\bf{H}_2, \bf{E}_2] & = -\bf{E}_2, & [\bf{H}_2, \bf{E}_3] & = 0, & [\bf{H}_2, \bf{E}_4] & = \bf{E}_4,\\
        [\bf{H}_1, \bf{F}_1] & = 0, & [\bf{H}_1, \bf{F}_2] & = -\bf{F}_2, & [\bf{H}_1, \bf{F}_3] & = -\bf{F}_3, & [\bf{H}_1, \bf{F}_4] & = -\bf{F}_4,\\
        [\bf{H}_2, \bf{F}_1] & = -\bf{F}_1, & [\bf{H}_2, \bf{F}_2] & = \bf{F}_2, & [\bf{H}_2, \bf{F}_3] & = 0, & [\bf{H}_2, \bf{F}_4] & = - \bf{F}_4,\\
        [\bf{E}_1, \bf{E}_2] & = \bf{E}_3, & [\bf{E}_1, \bf{E}_3] & = \bf{E}_4, & [\bf{E}_1, \bf{E}_4] & = 0, \\
        [\bf{E}_2, \bf{E}_3] & = 0, & [\bf{E}_2, \bf{E}_4] & = 0, & [\bf{E}_3, \bf{E}_4] & = 0, \\
        [\bf{F}_1, \bf{F}_2] & = -\bf{F}_3 & [\bf{F}_1, \bf{F}_3] & = -\bf{F}_4, & [\bf{F}_1, \bf{F}_4] & = 0, \\
        [\bf{F}_2, \bf{F}_3] & = 0, & [\bf{F}_2, \bf{F}_4] & = 0, & [\bf{F}_3, \bf{F}_4] & = 0, \\
        [\bf{E}_1, \bf{F}_1] & = \bf{H}_2, & [\bf{E}_1, \bf{F}_2] & = 0, & [\bf{E}_1, \bf{F}_3] & = -\bf{F}_2, & [\bf{E}_1, \bf{F}_4] & = - \bf{F}_3, \\
        [\bf{E}_2, \bf{F}_1] & = 0, & [\bf{E}_2, \bf{F}_2] & = \bf{H}_1 - \bf{H}_2, & [\bf{E}_2, \bf{F}_3] & = \bf{F}_1, & [\bf{E}_2, \bf{F}_4] & = 0, \\
        [\bf{E}_3, \bf{F}_1] & = -\bf{E}_2 & [\bf{E}_3, \bf{F}_2] & = \bf{E}_1, & [\bf{E}_3, \bf{F}_3] & = \bf{H}_1, & [\bf{E}_3, \bf{F}_4] & = \bf{F}_1, \\
        [\bf{E}_4, \bf{F}_1] & = - \bf{E}_3 & [\bf{E}_4, \bf{F}_2] & = 0, & [\bf{E}_4, \bf{F}_3] & = \bf{E}_1, & [\bf{E}_4, \bf{F}_4] & = \bf{H}_1 + \bf{H}_2. 
    \end{align*}
    Having in mind the classical PBW theorem for the universal enveloping algebra $U(\mathfrak{so}(5, \mathbb{C}))$ of $\mathfrak{so}(5, \mathbb{C})$, and since $U(\mathfrak{so}(5, \mathbb{C}))$ is a PBW extension of $\mathbb{C}$ \cite[Section 5]{BellGoodearl1988}, then $U(\mathfrak{so}(5, \mathbb{C}))$ is a skew PBW extension over $\mathbb{C}$, i.e., $U(\mathfrak{so}(5, \mathbb{C})) \cong \sigma(\mathbb{C})\langle \bf{J}_{\alpha \beta}\mid 1 \le \alpha \leq \beta \le 5 \rangle$. The weak annihilator of certain subsets, principal right ideals, and elements of algebra $U(\mathfrak{so}(5, \mathbb{C}))$ are characterized by Theorems \ref{theorem3.4}, \ref{Theorem3.1.6}, and \ref{theorem3.9}, respectively. Furthermore, Theorems \ref{theorem3.4.1}, \ref{Theorem3.13} and \ref{Theorem3.18} hold since $U(\mathfrak{so}(5, \mathbb{C}))$ is a skew PBW extension of endomorphism type. Theorem \ref{proposition3.4} asserts the existence of nilpotent good polynomials and Theorem \ref{Nilpotentprimes} describes the nilpotent associated prime ideals of $U(\mathfrak{so}(5, \mathbb{C}))$.

    \end{example}
    
    \begin{example}
    With the purpose of introducing generalizations of the classical bosonic and fermionic algebras of quantum mechanics concerning several versions of the Bose-Einstein and Fermi-Dirac statistics, Green \cite{Green1953} and Greenberg and Messiah \cite{GreenbergMessiah1965} introduced by means of generators and relations the {\em parafermionic} and {\em parabosonic algebras}. For the completeness of the paper, briefly we recall the definition of each one of these structures following the treatment developed by Kanakoglou and Daskaloyannis \cite{KanakoglouDaskaloyannis2009}. Let $[\square, \triangle] := \square \triangle - \triangle \square$ and $\{\square, \triangle\} := \square \triangle + \triangle \square$.

   Consider the $\Bbbk$-vector space $V_F$ freely generated by the elements $f_i^{+}, f_j^{-}$, with $i, j = 1,\dotsc, n$. If $T(V_F)$ is the tensor algebra of $V_F$ and $I_F$ is the two-sided ideal $I_F$ generated by the elements $[[f_i^{\xi}, f_j^{\eta}], f_k^{\varepsilon}] - \frac{1}{2}(\varepsilon - \eta)^{2} \delta_{jk} f_i^{\xi} + \frac{1}{2}(\varepsilon - \xi)^{2}\delta_{ik}f_j^{\eta}$, for all values of $\xi, \eta, \varepsilon = \pm 1$, and $i, j, k = 1,\dotsc, n$, then the {\em parafermionic algebra} in $2n$ generators $P_F^{(n)}$ ($n$ parafermions) is the quotient algebra of $T(V_F)$ with the ideal $I_F$, that is, {\small{
   \[
   P_F^{(n)} = \frac{T(V_F)}{\langle [[f_i^{\xi}, f_j^{\eta}], f_k^{\varepsilon}] - \frac{1}{2}(\varepsilon - \eta)^{2} \delta_{jk} f_i^{\xi} + \frac{1}{2}(\varepsilon - \xi)^{2}\delta_{ik}f_j^{\eta} \mid \xi, \eta, \varepsilon = \pm 1, i, j, k = 1,\dotsc, n\rangle}
   \]
   }}
   It is well-known (e.g., \cite[Section 18.2]{KanakoglouDaskaloyannis2009}) that a parafermionic algebra $P_F^{(n)}$ in $2n$ generators is isomorphic to the universal enveloping algebra of the simple complex Lie algebra $\mathfrak{so}(2n+1)$ (according to the classification of the simple complex Lie algebras, e.g., Kac \cite{Kac1977, Kac1977a}), i.e., $P_F^{(n)} \cong U(\mathfrak{so}(2n+1))$. On the other hand, $P_F^{(n)}$ is a skew PBW extension over $\Bbbk$, that is,  $P_F^{(n)} \cong \sigma(\Bbbk)\langle f_i^{\xi}, f_j^{\eta}\rangle$, for values $\xi, \eta = \pm 1$, and $1 \le i, j \le n$. Since $\Bbbk$ is a field (hence reduced), then $r_{\Bbbk}(X)=N_{\Bbbk}(X)=\left \{ 0 \right \}$, for all $X\nsubseteq N(\Bbbk)$. This implies that the characterization the weak annihilators of $N_{P_F^{(n)}}(U)$, $N_{P_F^{(n)}}(f\ P_F^{(n)})$, and $N_{P_F^{(n)}}(f)$, for some subsets, principal right ideals, and elements of algebra $P_F^{(n)}$, follow from Theorems \ref{theorem3.4}, \ref{Theorem3.1.6}, and \ref{theorem3.9}, respectively. Furthermore, Theorems \ref{theorem3.4.1}, \ref{Theorem3.13} and \ref{Theorem3.18} are also valid for the description of these weak annihilators. We recall that these subsets, principal right ideals and elements, are not contained in the set of nilpotent elements of the algebra $P_F^{(n)}$. For the parafermionic algebra, Theorem \ref{proposition3.4} guarantees the existence of nilpotent good polynomials and Theorem \ref{Nilpotentprimes} describes the nilpotent associated prime ideals.
    
  \vspace{0.2cm}
    
    Similarly, if $V_B$ denotes the $\Bbbk$-vector space freely generated by the elements $b_i^{+}, b_j^{-}$, $i, j = 1,\dotsc, n$, $T(V_B)$ is the tensor algebra of $V_B$, and $I_B$ is the two-sided ideal of $T(V_B)$ generated by the elements $[\{b_i^{\xi}, b_j^{\eta}\}, b_k^{\varepsilon}] - (\varepsilon - \eta)\delta_{jk}b_i^{\xi} - (\varepsilon - \xi)\delta_{ik} b_j^{\eta}$, for all values of $\xi, \eta, \varepsilon = \pm 1$, and $i, j = 1, \dotsc, n$, then the {\em parabosonic algebra} $P_B^{(n)}$ in $2n$ generators ($n$ parabosons) is defined as the quotient algebra $P_B^{(n)} / I_B$, that is,
    \[
    P_B^{(n)} = \frac{T(V_B)}{\langle [\{b_i^{\xi}, b_j^{\eta}\}, b_k^{\varepsilon}] - (\varepsilon - \eta)\delta_{jk}b_i^{\xi} - (\varepsilon - \xi)\delta_{ik} b_j^{\eta}\mid \xi, \eta, \varepsilon = \pm 1, i, j = 1, \dotsc, n \rangle }
    \]
    It is known that the parabosonic algebra $P_B^{(n)}$ in $2n$ generators is isomorphic to the universal enveloping algebra of the classical simple complex Lie superalgebra $B(0,n)$, that is, $P_B^{(n)}\cong U(B(0,n))$. For more details about {\em parafermionic} and {\em parabosonic algebras}, see \cite[Proposition 18.2]{KanakoglouDaskaloyannis2009} and references therein.

    Similar to the case of parafermionic algebra, the parabosonic algebra is a skew PBW extension over $\Bbbk$, that is, $P_B^{(n)} \cong \sigma(\Bbbk)\langle b_i^{\xi}, b_j^{\eta}\rangle$ for values $\xi, \eta = \pm 1$, and $1 \le i, j \le n$. The results mentioned for the case of parafermionic algebras hold for parabasonic algebras. The weak annihilator of some subsets, principal right ideals, and elements of algebra $P_B^{(n)}$ are characterized by Theorems \ref{theorem3.4}, \ref{Theorem3.1.6}, and \ref{theorem3.9}, respectively. Furthermore, Theorems \ref{theorem3.4.1}, \ref{Theorem3.13} and \ref{Theorem3.18} are valid since $P_B^{(n)}$ is a skew PBW extension of endomorphism type. Finally, thinking about the theory of nilpotent associated prime ideals, Theorem \ref{proposition3.4} guarantees the existence of nilpotent good polynomials, and Theorem \ref{Nilpotentprimes} describes the nilpotent associated prime ideals of $P_B^{(n)}$.
   
    \end{example}

\begin{example}
Other algebraic structures that illustrate the results obtained in the paper concerns examples of {\em generalized Weyl algebras}, {\em down-up algebras}, and {\em ambiskew polynomial rings}. For the completeness of the paper, we will briefly present the definitions and some relations between these algebras (see \cite{Jordan1995, Jordan1995b, Jordan2000, JordanWells1996} for a detailed description).

Given an automorphism $\sigma$ and a central element $a$ of a ring $R$,  Bavula \cite{Bavula92} defined the {\em generalized Weyl algebra} $R(\sigma, a)$ as the ring extension of $R$ generated by the indeterminates $X^{-}$ and $X^{+}$ subject to the relations
\begin{equation*}
    X^{-}X^{+} = a,\qquad X^{+}X^{-} = \sigma(a),
\end{equation*}
and, for all $b\in R$,
\begin{equation*}
X^{+}b = \sigma(b)X^{+},\qquad X^{-}\sigma(b) = bX^{-}.
\end{equation*}
This family of algebras includes the classical Weyl algebras, primitive quotients of $U(\mathfrak{sl}_2)$, and ambiskew polynomial rings (see below the definition of these objects). Generalized Weyl algebras have been extensively studied in the literature by various authors (see \cite{Bavula92, Bavula1996, Bavula2021, Jordan2000}, and references therein).

\medskip

The {\em down-up algebras} $A(\alpha, \beta, \gamma)$, where $\alpha, \beta, \gamma \in \mathbb{C}$, were defined by Benkart and Roby \cite{Benkart1998, BenkartRoby1998} as generalizations of algebras generated by a pair of operators, precisely, the \textquotedblleft down\textquotedblright\ and \textquotedblleft up\textquotedblright\ operators, acting on the vector space $\mathbb{C}P$ for certain partially ordered sets $P$. Let us see the details.

Consider a partially ordered set $P$ and let $\mathbb{C} P$ be the complex vector space with basis $P$. If for an element $p$ of $P$, the sets $\{x\in P\mid x \succ p\}$ and $\{x\in P\mid x\prec p\}$ are finite, then we can define the \textquotedblleft down\textquotedblright\ operator $d$ and the \textquotedblleft down\textquotedblright\ operator $u$ in ${\rm End}_{\mathbb{C}} \mathbb{C}P$ as $u(p) = \sum_{x\succ p} x$ and $d(p) = \sum_{x \prec p} x$, respectively (for partially ordered sets in general, one needs to complete $\mathbb{C}P$ to define $d$ and $u$). For any $\alpha, \beta, \gamma \in \mathbb{C}$, the {\em down-up algebra} is the $\mathbb{C}$-algebra generated by $d$ and $u$ subject to the relations 
\begin{align}
    d^2 u = \alpha dud + \beta ud^2 + \gamma d, \label{(1a)}  \\
    du^2 = \alpha udu + \beta u^2 d + \gamma u. \label{(1b)}
\end{align}
A partially ordered set $P$ is called $(q, r)$-{\em differential} if there exist $q, r\in \mathbb{C}$ such that the down and up operators for $P$ satisfy the relations (\ref{(1a)}) and (\ref{(1b)}), and $\alpha = q(q+1), \beta = -q^3$, and $\gamma = r$. From \cite{BenkartRoby1998}, we know that for $0\neq \lambda \in \mathbb{C}$, $A(\alpha, \beta, \gamma) \simeq A(\alpha, \beta, \lambda \gamma)$. This means that when $\gamma \neq 0$, no problem if we assume $\gamma = 1$. For more details about the combinatorial origins of down-up algebras, see \cite[Section 1]{Benkart1998}.

Remarkable examples of down-up algebras include the enveloping algebra of the Lie algebra $\mathfrak{sl}_2(\mathbb{C})$ and some of its deformations introduced by Witten \cite{Witten1990} and  Woronowicz \cite{Woronowicz1987}. Related to the theoretical properties of these algebras, Kirkman et al. \cite{Kirkmanetal1999} proved that a down-up algebra $A(\alpha, \beta, \gamma)$ is Noetherian if and only if $\beta$ is non-zero. As a matter of fact, they showed that $A(\alpha, \beta, \gamma)$ is a generalized Weyl algebra in the sense of Bavula \cite{Bavula92}, and that $A(\alpha, \beta, \gamma)$ has a filtration for which the associated graded ring is an iterated Ore extension over $\mathbb{C}$.

Following \cite[p. 32]{Benkart1998}, if $\mathfrak{g}$ is a 3-dimensional Lie algebra over $\mathbb{C}$ with basis $x, y, [x, y]$ such that $[x, [x, y]] = \gamma x$ and $[[x, y], y] = \gamma y$, then in the universal enveloping algebra $U(\mathfrak{g})$ of $\mathfrak{g}$ these relations are given by
\begin{align*}
    x^2y - 2xyx + yx^{2} = &\ \gamma x,\\
    xy^2 - 2yxy + y^2x = &\ \gamma y.
\end{align*}
Notice that $U(\mathfrak{g})$ is a homomorphic algebra of the down-up algebra $A(2, -1, \gamma)$ via the mapping $\phi: A(2, -1, \gamma) \to U(\mathfrak{g})$, $d\mapsto x, u\mapsto y$, and the mapping $\psi: \mathfrak{g} \to A(2, -1,\gamma)$, $x\mapsto d, y\mapsto u$, $[x, y]\mapsto du-ud$, extends by the universal property of $U(\mathfrak{g})$ to an algebra homomorphism $\psi: U(\mathfrak{g}) \to A(2, -1, \gamma)$ which is the inverse of $\psi$, then $U(\mathfrak{g})$ is isomorphic to $A(2, -1, \gamma)$.

The well-known Lie algebra $\mathfrak{sl}_2$ of $2\times 2$ complex matrices of trace zero has a standard basis $e, f, h$, which satisfies $[e, f] = h,\ [h, e] = 2e$, and $[h, f] = -2f$. It is straightforward to see that $U(\mathfrak{sl}_2)\cong A(2, -1, -2)$. Also, for the Heisenberg Lie algebra $\mathfrak{h}$ with basis $x, y, z$ where $[x, y] = z$ and $[z, \mathfrak{h}] = 0$, $U(\mathfrak{h}) \cong A (2, -1, 0)$.

With the aim of providing an explanation of the existence of quantum groups, Witten \cite{Witten1990, Witten1991} introduced a 7-parameter deformation of the universal enveloping algebra $U(\mathfrak{sl}_2)$. By definition, Witten's deformation is a unital associative algebra over a field $\Bbbk$ (which is algebraically closed of characteristic zero) that depends on a $7$-tuple $\underline{\xi} = (\xi_1, \dotsc, \xi_7)$ of elements of $\Bbbk$. This algebra is generated by the indeterminates $x, y, z$ subject to the defining relations
\begin{align}
    xz - \xi_1 zx = &\ \xi_2 x, \label{Witten(2.1)}\\
    zy - \xi_3 yz = &\ \xi_4 y, \label{Witten(2.2)} \\
    yx - \xi_5xy = &\ \xi_6z^{2} + \xi_7z, \label{Witten(2.3)}
\end{align}
and is denoted by $W(\underline{\xi})$. By \cite[Section 2]{Benkart1998}, if $\xi_6 = 0$ and $\xi_7\neq 0$ then it is easy to see that expressions (\ref{Witten(2.1)}), (\ref{Witten(2.2)}), and (\ref{Witten(2.3)}) imply the equalities
\begin{align*}
    -\xi_5 x^2y + (1+\xi_1\xi_5)xyx - \xi_1yx^{2} = &\ \xi_2\xi_7x,\\
    -\xi_5xy^{2} + (1+\xi_3\xi_5)yxy - \xi_3y^2x = &\ \xi_4\xi_7y.
\end{align*}
When $\xi_5 \neq 0, \xi_1 = \xi_3$, and $\xi_2 = \xi_4$, we obtain 
\begin{align*}
    x^2y = &\ \frac{1+\xi_1\xi_5}{\xi_5}xyx - \frac{\xi_1}{\xi_5}yx^{2} - \frac{\xi_2\xi_7}{\xi_5}x, \\
    xy^{2} = &\ \frac{1 + \xi_1\xi_5}{\xi_5}yxy - \frac{\xi_1}{\xi_5}y^2x - \frac{\xi_2\xi_7}{\xi_5}y.
\end{align*}
As one can check, a Witten deformation algebra $W(\underline{\xi})$ with $\xi_6 = 0, \xi_5 \xi_7 \neq 0 , \xi_1 = \xi_3$, and $\xi_2 = \xi_4$ is a homomorphic image of the down-up algebra $A(\alpha, \beta, \gamma)$ where 
\begin{equation}\label{Benkart1998(2.5)}
\alpha = \frac{1 + \xi_1 \xi_5}{\xi_5},\quad \beta = -\frac{\xi_1}{\xi_5},\quad \gamma = -\frac{\xi_2\xi_7}{\xi_5}.
\end{equation}
All these facts allow to prove that a Witten deformation algebra $W(\xi)$ with 
\begin{equation}\label{Benkart1998(2.7)}
\xi_6 = 0,\quad \xi_5\xi_7 \neq 0,\quad \xi_1 = \xi_3,\quad {\rm and}\quad \xi_2 = \xi_4,     
\end{equation}
is isomorphic to the down-up algebra $A(\alpha, \beta, \gamma)$ with $\alpha, \beta, \gamma$ given by  (\ref{Benkart1998(2.5)}). Notice that any down-up algebra $A(\alpha, \beta, \gamma)$ with not both $\alpha$ and $\beta$ equal to $0$ is isomorphic to a Witten deformation algebra $W(\underline{\xi})$ whose parameters satisfy (\ref{Benkart1998(2.7)}).

Since algebras $W(\underline{\xi})$ are filtered, Le Bruyn \cite{LeBruyn1994, LeBruyn1995} studied the algebras $W(\underline{\xi})$ whose associated graded algebras are Auslander regular. He determined a 3-parameter family of deformation algebras which are said to be {\em conformal} $\mathfrak{sl}_2$ {\em algebras} that are generated by the indeterminates $x, y, z$ over a field $\Bbbk$ subject to the relations given by
\begin{equation}\label{Benkart1998(2.10)}
    zx - axz = x,\quad zy - ayz = y, \quad yx - cxy = bz^2 + z.
\end{equation}
In the case $c\neq 0$ and $b = 0$, the conformal $\mathfrak{sl}_2$ algebra with defining relations given by (\ref{Benkart1998(2.10)}) is isomorphic to the down-up algebra $A(\alpha, \beta, \gamma)$ with $\alpha = c^{-1}(1 + ac), \beta = -ac^{-1}$ and $\gamma = -c^{-1}$. Notice that if $c = b = 0$ and $a \neq 0$, then the conformal $\mathfrak{sl}_2$ algebra is isomorphic to the down-up algebra $A(\alpha, \beta, \gamma)$ with $\alpha = a^{-1}, \beta = 0$, and $\gamma = -a^{-1}$. As one can check using (\ref{Benkart1998(2.10)}), conformal $\mathfrak{sl}_2$ algebras are not Ore extensions but skew PBW extensions over $\Bbbk[z]$. 

Of interest for the examples in the paper, Kulkarni \cite{Kulkarni1999} showed that under certan assumptions on the parameters, a Witten deformation algebra is isomorphic to a conformal $\mathfrak{sl}_2$ algebra or to an iterated Ore extension (double Ore extension). More exactly, following \cite[Theorem 3.0.3]{Kulkarni1999} if $\xi_1\xi_3\xi_5\xi_2 \neq 0$ or $\xi_1\xi_3\xi_5\xi_4 \neq 0$, then $W(\underline{\xi})$ is isomorphic to one of the following algebras:
\begin{enumerate}
    \item [\rm (1)] A conformal $\mathfrak{sl}_2$ algebra with generators $x, y, z$ and relations given by (\ref{Benkart1998(2.10)}), for some elements $a, b, c\in \Bbbk$.
    \item [\rm (2)] An iterated Ore extension whose generators satisfy
    \begin{enumerate}
        \item [\rm (i)] $xz - zx = x$, $zy - yz = \zeta y$, $yx - \eta xy = 0$, or
        \item [\rm (ii)] $xw = \theta wx$, $wy = \kappa yw$, $yx = \lambda xy$, for parameters $\zeta, \eta, \theta, \kappa, \lambda \in \Bbbk$.
    \end{enumerate}
\end{enumerate}
Considering the above facts and theorems presented in this paper, we can assert that weak annihilators of subsets, principal right ideals, and elements of conformal $\mathfrak{sl}_2$ algebras are characterized by Theorems \ref{theorem3.4}, \ref{Theorem3.1.6}, and \ref{theorem3.9}, respectively. About the theory of nilpotent associated prime ideals, Theorem \ref{proposition3.4} establishes the existence of nilpotent good polynomials, and Theorem \ref{Nilpotentprimes} describes the nilpotent associated prime ideals of the algebra. 


\medskip

Jordan \cite{Jordan2000} introduced a certain class of iterated Ore extensions $R(B, \sigma, c, p)$ called {\em ambiskew polynomial rings} (these structures have been studied by Jordan at various levels of generality in several papers \cite{Jordan1993, Jordan1993b, Jordan1995, Jordan1995b}) that contains different examples of noncommutative algebras. We recall the treatment of ambiskew polynomial rings in the framework appropriate to down-up algebras with $\beta \neq 0$ following \cite{Jordan2000, JordanWells1996}.

For $\Bbbk$ a field and $\Bbbk^{*} = \Bbbk\ \backslash\ \{0\}$ its multiplicative group, consider a commutative $\Bbbk$-algebra $B$, a $\Bbbk$-automorphism of $B$, and elements $c\in B$ and $p\in \Bbbk^{*}$. Let $S$ be the Ore extension $B[x;\sigma^{-1}]$ and extend $\sigma$ to $S$ by setting $\sigma(x) = px$. By \cite[p. 41]{Cohn1985}, there is a $\sigma$-derivation $\delta$ of $S$ such that $\delta(B) = 0$ and $\delta(x) = c$. The {\em ambiskew polynomial ring} $R = R(B,\sigma, c, p)$ is the Ore extension $S[y;\sigma,\delta]$, whence the following relations hold:
\begin{equation}
    yx - pxy = c,\quad {\rm and,\ for\ all}\ b\in B,\quad xb = \sigma^{-1}(b)x\quad {\rm and}\quad yb = \sigma(b)y.
\end{equation}

Equivalently, $R$ can be presented as $R = B[y;\sigma][x;\sigma^{-1},\delta']$ with $\sigma(y) = p^{-1}y$, $\delta'(B) = 0$, and $\delta'(y) = -p^{-1}c$, so that $xy - p^{-1}yx = -p^{-1}c$. If we consider the relation $xb = \sigma^{-1}(b)x$ as $bx = x\sigma(b)$, then we can see that the definition involves twists from both sides using $\sigma$; this is the reason for the name of the objects. As a matter of fact, every generalized Weyl algebra is isomorphic to a factor of an ambiskew polynomial ring \cite{Jordan1993}, and the ambiskew polynomial ring $A(R, \sigma, c, p)$ is isomorphic to the generalized Weyl algebra $R[w](\sigma, w)$, where $w$ is extended to $R[w]$ by setting $\sigma(w) = pw+\sigma(c)$. In the conformal case, $A(R, \sigma, c, p) \simeq R[z](\sigma, z+\sigma(a))$, where $\sigma(z) = pz$. From the definition, it follows that ambiskew polynomial rings are skew PBW extensions over $B$, that is, $R(B, \sigma, c, p)\cong \sigma(B)\langle y, x\rangle$. 
\end{example}

\begin{example}
Recently, Bavula \cite{Bavula2021} defined the {\em skew bi-quadratic algebras} with the aim of giving an explicit description of bi-quadratic algebras on 3 generators with PBW basis. Let us recall the details of its definition.

For a ring $R$ and a natural number $n\ge 2$, a family $M = (m_{ij})_{i > j}$ of elements $m_{ij}\in R$ ($1\le j < i \le n$) is called a {\em lower triangular half-matrix} with coefficients in $R$. The set of all such matrices is denoted by $L_n(R)$.

If $Z(R)$ denotes the center of $R$, $\sigma = (\sigma_1, \dotsc, \sigma_n)$ is an $n$-tuple of commuting endomorphisms of $R$, $\delta = (\delta_1, \dotsc, \delta_n)$ is an $n$-tuple of $\sigma$-endomorphisms of $R$ (that is, $\delta_i$ is a $\sigma_i$-derivation of $R$ for $i=1,\dotsc, n$), $Q = (q_{ij})\in L_n(Z(R))$, $\mathbb{A}:= (a_{ij, k})$ where $a_{ij, k}\in R$, $1\le j < i \le n$ and $k = 1,\dotsc, n$, and $\mathbb{B}:= (b_{ij})\in L_n(R)$, the {\em skew bi-quadratic algebra}) ({\em SBQA}) $A = R[x_1,\dotsc, x_n;\sigma, \delta, Q, \mathbb{A}, \mathbb{B}]$ is a ring generated by the ring $R$ and elements $x_1, \dotsc, x_n$ subject to the defining relations
\begin{align}
    x_ir = &\ \sigma_i(r)x_i + \delta_i(r),\quad {\rm for}\ i = 1, \dotsc, n,\ {\rm and\ every}\ r\in R, \label{Bavula2021(1)} \\
    x_ix_j - q_{ij}x_jx_i = &\ \sum_{k=1}^{n} a_{ij, k}x_k + b_{ij},\quad {\rm for\ all}\ j < i.\label{Bavula2021(2)}
\end{align}
In the particular case, when $\sigma_i = {\rm id}_R$ and $\delta_i = 0$, for $i = 1,\dotsc, n$, the ring $A$ is called the {\em bi-quadratic algebra} ({\em BQA}) and is denoted by $A = R[x_1, \dotsc, x_n; Q, \mathbb{A}, \mathbb{B}]$. $A$ has {\em PBW basis} if $A = \bigoplus_{\alpha \in \mathbb{N}^{n}} Rx^{\alpha}$ where $x^{\alpha} = x_1^{\alpha_1}\dotsb x_n^{\alpha_n}$.
\end{example}
It is clear from the definition that $R[x_1, \dotsc, x_n;\sigma, \delta, Q, \mathbb{A}, \mathbb{B}] \cong \sigma(R)\langle x_1,\dotsc, x_n\rangle$, and, as a matter of fact, skew PBW extensions are more general than skew bi-quadratic algebras since the former do not assume the commutativity of $\sigma$'s (c.f. \cite[Definition 2.1]{LezamaAcostaReyes2015}) and the conditions about the membership of the elements $q_{ij}$ and $b_{ij}$. 

Notice that if $R$ is a $(\Sigma,\Delta)$-compatible NI ring, then the characterization of the weak annihilators $N_{A}(U)$, $N_{A}(fA)$, and $N_{A}(f)$, for subsets, principal right ideals, and elements of algebra $A= R[x_1, \dotsc, x_n; Q, \mathbb{A}, \mathbb{B}]$, respectively, follow from Theorems \ref{theorem3.4}, \ref{Theorem3.1.6}, and \ref{theorem3.9}. If, in addition, $R$ is a field, Theorems \ref{theorem3.4.1}, \ref{Theorem3.13} and \ref{Theorem3.18} are also valid for the description of all of them. Finally, for the bi-quadratic algebras, if $R$ is a $(\Sigma,\Delta)$-compatible NI ring, then Theorem \ref{proposition3.4} guarantees the existence of nilpotent good polynomials, and Theorem \ref{Nilpotentprimes} describes its nilpotent associated prime ideals.

\section{Future work}

In this paper, we have studied weak annihilators and nilpotent associated primes of skew PBW extensions over $(\Sigma,\Delta)$-compatible rings, where some results were shown in the more general setting of weak compatibility. Having in mind that Louzari, Ni\~no, Ram\'irez and Reyes \cite{LouzariReyes2020Springer, NinoRamirezReyes2020, Reyes2019} considered the notion of compatibility (and relater ring-theoretical properties) for modules over these extensions, a first and natural task is to characterize nilpotent associated prime ideals over these objects.

\medskip

Related with this, since Macdonald \cite{Macdonald1973} introduced a dual theory to primary decomposition called {\em secondary representation} whose key objects are known as {\em attached primes}, which were considered by Annin \cite{Annin2008, Annin2011} with the aim of studying the behavior of the attached prime ideals of inverse polynomial modules over skew polynomial rings of automorphism type $R[x;\sigma]$, a second task is to determine the attached primes of polynomial modules over skew PBW extensions with the aim of generalizing Annin's results and others related in the Module and Ring Theory.


\end{document}